\documentclass[11pt, reqno]{amsart}
\usepackage{enumerate}
  \usepackage{geometry }
 \usepackage{amsmath}
 \usepackage{amsfonts}
 \usepackage{amssymb}
  \usepackage[utf8]{inputenc} 
  \usepackage{enumerate}
 \usepackage[all]{xy}
\usepackage{amssymb,amsmath,amscd}
\usepackage{amsthm}

\usepackage{color}
 \usepackage{mathtools}
 \usepackage{enumerate}

 \numberwithin{equation}{section}

   \newtheorem{thmalph}{Theorem}

 \newtheorem{lemmab}[thmalph]{Lemma}

\newtheorem{theorem}{Theorem}[section]
 
\newtheorem*{theorem*}{Theorem }
\newtheorem{proposition}[theorem]{Proposition}%[section]

\newtheorem{corollary}[theorem]{Corollary}%[section]

\newtheorem{lemma}[theorem]{Lemma}%[section]
\newtheorem{remark}[theorem]{Remark}%[section]

 \usepackage{mathtools}
\usepackage[bbgreekl]{mathbbol}
\usepackage{amsfonts}

\DeclareSymbolFontAlphabet{\mathbb}{AMSb}
 \usepackage[
 pagebackref,  
 colorlinks=true,  citecolor=cyan,   
 citebordercolor={0 .255 .255} , 
  linkbordercolor={0 .255 .255},  
  linkcolor=cyan
 ]{hyperref}

   \usepackage{cite}

 \def\equationautorefname~#1\null{(#1)\null}
 
\usepackage{multirow}
  \usepackage{graphicx}
 \usepackage{caption}
 \usepackage{booktabs}
 \usepackage{threeparttable}
\newcommand{\C}{\mathbb C}
\newcommand{\R}{\mathbb R}

\newcommand*\diff{\mathop{}\!\mathrm{d}}

      \DeclareMathOperator{\SO}{SO}
            
\usepackage{slashed}
\makeatletter
\newcommand{\extp}{\@ifnextchar^\@extp{\@extp^{\,}}}
\def\@extp^#1{\mathop{\bigwedge\nolimits^{\!#1}}}
\makeatother

      \usepackage[normalem]{ulem}
      \usepackage{color}
        \usepackage{xcolor}
      \usepackage{cancel}

\newcommand\blfootnote[1]{%
  \begingroup
  \renewcommand\thefootnote{}\footnote{#1}%
  \addtocounter{footnote}{-1}%
  \endgroup
}
\begin{document}

\title[Generalized spectral projections for the Hodge-de Rham Laplacian]
{A characterization of the $L^2$-range of the generalized spectral projections related to the Hodge-de Rham Laplacian
}

\author{Abdelhamid Boussejra}
\address{Abdelhamid Boussejra : Department of Mathematics, Faculty of Sciences, University Ibn Tofail, Kenitra, Morocco 
}
\curraddr{}
\email{boussejra.abdelhamid@uit.ac.ma}
\thanks{}

\author{Khalid Koufany}
\address{Khalid Koufany : Université de Lorraine, CNRS, IECL, F-54000 Nancy, France}
 
\curraddr{}
\email{khalid.koufany@univ-lorraine.fr}
\thanks{}
\date{\today}
 
 \maketitle

\begin{abstract}  %Version \today\\
    In this paper, we establish  a characterization of the $L^2$-range of generalized spectral projections on the bundle of differential forms over the real hyperbolic space $ H^n(\mathbb R)$. As an  intermediate  result, we obtain  a characterization of the $L^2$-range of the Poisson transform on the bundle of differential forms on the boundary $\partial H^n(\mathbb R)$. This   results confirm    a conjecture by  Strichartz  regarding  differential forms.

  \end{abstract}

\blfootnote{\emph{2010 AMS Classification : \rm 43A90, 43A85, 58J50}} 

\tableofcontents

  \section{Introduction}

  \sloppy
  
 Let $T^*_{\mathbb C} H^n(\mathbb R)$ be the complexified cotangent bundle of   the real hyperbolic space $H^n(\mathbb R)$ and  let  $\extp^p  H^n(\mathbb R):= \extp^p T^*_{\mathbb C} H^n(\mathbb R)$ be the $p$-th exterior vector bundle ( $0\leqslant p\leqslant n$). Its sections are  differential $p$-forms on $H^n(\mathbb R)$.  
 We write $C^\infty( \extp^p  H^n(\mathbb R))$, $C_c^\infty( \extp^p  H^n(\mathbb R))$ and $L^2( \extp^p  H^n(\mathbb R))$ for  smooth sections, smooth sections with compact support and square-integrable sections, respectively. Consider   the exterior differentiation operator 
$
 d \colon C^\infty (\extp^p  H^n(\mathbb R)) \to C^\infty (\extp^{p+1}  H^n(\mathbb R))
 $
 and the co-differentiation operator    
 $
 \delta \colon  C^\infty (\extp^{p+1}  H^n(\mathbb R)) \to  C^\infty( \extp^{p}  H^n(\mathbb R))
 $. 
 Let 
$\Delta =d\delta+\delta d$ be the Hodge-de Rham Laplacian
which is a self-adjoint operator on   $L^2(\extp^p  H^n(\mathbb R))$.
 According to  the Hodge-de Rahm  decomposition of $L^2(\extp^p  H^n(\mathbb R))$ any $f\in L^2(\extp^p  H^n(\mathbb R))$ can be decomposed into a sum 
\begin{align}\label{hodge}
f=f_d+f_\delta+f_0,
\end{align}
 where $f_d\in  \overline{d C^\infty_c(\extp^{p-1}  H^n(\mathbb R))},$ $f_\delta\in \overline{\delta C^\infty_c(\extp^{p+1}  H^n(\mathbb R))}$,  and $f_0$ is an $L^2$-harmonic form. The  harmonic component $f_0$  is  non zero only   when $n$ is even and $p=n / 2$.
The orthogonal decomposition \eqref{hodge} yields to three projectors (de-Rham projectors). Following Strichartz we are interested in the first two.
More precisely, in \cite{strichartz}, Strichartz  initiated  the spectral theory of the Hodge-de Rham Laplacian  $\Delta$  by considering the decomposition of forms into exact differentials ($d$) and coexact differentials ($\delta$).   In the following, we will briefly  outline his approach.  
% and is related to   discrete series representations of $\SO_0(n, 1)$. 
 By using a  local presentation  of $p$-forms, Strichartz found  explicit expressions for operators $\mathcal{Q}_\lambda^d$ and $\mathcal{Q}_\lambda^\delta$  which satisfy  
$$
\begin{aligned}
f_d&=\int_0^{\infty} \mathcal{Q}_\lambda^d f \diff \lambda, \\
f_\delta&=\int_0^{\infty} \mathcal{Q}_\lambda^\delta f \diff \lambda,
\end{aligned}
$$
  with
$$
\begin{aligned}
\Delta \mathcal{Q}_\lambda^d f&=-(\lambda^2+(\rho-p+1)^2)) \mathcal{Q}_\lambda^d f,\\
\Delta \mathcal{Q}_\lambda^{\delta} f&=-(\lambda^2+(\rho-p)^2) \mathcal{Q}_\lambda^{\delta} f.
\end{aligned}
$$
He   noticed   that for a dense set of $f$ and for any point $z$,
  \begin{equation}\label{stri-int1}
  \left\|f_\delta\right\|_2^2=\pi \lim _{r \rightarrow \infty} \int_0^\infty \frac{1}{r} \int_{B_r(z)}\left\langle\mathcal{Q}_\lambda^\delta f, \mathcal{Q}_\lambda^\delta f\right\rangle_x \diff  x \diff  \lambda,
  \end{equation}
and
\begin{equation}\label{stri-int2}
\left\|f_d\right\|_2^2=\pi \lim _{r \rightarrow \infty} \int_0^\infty \frac{1}{r} \int_{B_r(z)}\left\langle\mathcal{Q}_\lambda^d f, \mathcal{Q}_\lambda^d f\right\rangle_x \diff x \diff \lambda.
	\end{equation}
 Then he made the following  conjecture,  {\it loc.\,cit.} \cite[page 146]{strichartz}  :{\it \lq\lq It is natural to conjecture that
 \eqref{stri-int1} and \eqref{stri-int2}   continue to hold for all $L^2$, and conversely, the finiteness of \eqref{stri-int1} and \eqref{stri-int2} should characterize   the spectral decomposition of $L^2$ forms\rq\rq.}

 In this paper, we focus on proving the aforementioned conjecture. Our result will extend the \lq\lq classical\rq\rq\, Strichartz's conjecture (refer to \cite[Conjecture 4.5, Conjecture 4.6]{strichartz})) from scalar-valued functions on Riemannian symmetric spaces of noncompact type to the differential forms bundle on $H^n(\mathbb R)$. We should mention that  \cite[Conjecture 4.5, Conjecture 4.6]{strichartz}  has been investigated for rank one symmetric spaces (see e.g. \cite{BS, B-Nadia, IB, Obray, Ionescu}) and proved for any symmetric space of noncompact type in \cite{Kaizuka}.

  The method employed  by Strichartz for \eqref{stri-int1} and \eqref{stri-int2}  has the disadvantage of being local: he realized $H^n(\mathbb R)$ in $\mathbb R^{n+1}$ and used local harmonic analysis to define $\mathcal Q_\lambda^d$  and $\mathcal Q_\lambda^\delta$,    and to prove \eqref{stri-int1} and \eqref{stri-int2}  through  direct computations.   
 
 Our  approach  for handling this conjecture is based on representation theory. We will view $H^n(\mathbb{R})$ as the rank one symmetric space $G/K$ where $G = \SO_0(n,1)$  and $K = \SO(n)$ its maximal compact subgroup.
 As working in the setting of the exterior bundle involves some  difficulties, we shall identify $\extp^p  H^n(\mathbb R)$ with the  $G$-homogeneous vector bundle $G\times_K \bigwedge ^p\mathbb C^n$ and   subsequently  use representation theory, in particular  the Plancherel decomposition of $L^2(G\times_K \bigwedge ^p\mathbb C^n)$. To be more precise,
let $\mathfrak{g}=\mathfrak{k}\oplus \mathfrak{p}$ be the Cartan decomposition and let $\mathfrak{a}\subset \mathfrak{p}$ be a maximal abelian subspace of $\mathfrak{p}$. Let   $\tau=\tau_p=\extp^p Ad^\ast$ ($0\leqslant p\leqslant n$)   be the $p$-th exterior product of the coadjoint representation of $K$ on  $V_\tau=\extp^p \mathfrak{p}^\ast_{\mathbb C}\simeq \extp^p \mathbb C^n$ (Notice that $\tau_p$ is irreducible if $p\neq \frac{n}{2}$). Then as $G$-homogeneous vector bundles we have $\extp^p T^\ast_\mathbb{C} H^n(\mathbb{R})=G\times_K V_\tau$. So we shall identify the space $L^2(\extp^p  H^n(\mathbb R))$   with the space $L^2(G,\tau)$ of square integrable function $f:G\to V_\tau$, such that $f(gk)=\tau(k)^{-1}f(g)$ for $g\in G$ and $k\in K$.
 According to the Plancherel Theorem, the space  $L^2(G,\tau)$ splits into a direct sum of a continuous part $L^2(G,\tau)_{\mathrm{cont}}$ and possibly of a discrete part $L^2(G,\tau)_{\mathrm{disc}}$. The latter is non-trivial if $p=\frac{n}{2}$. 
  
Due  to the Helgason Fourier inversion formula (see e.g. \cite{Pedon}), any  $f\in L^2(G,\tau)$ (or in the continuous part for $p=\frac{n}{2}$)  can be written as 
 $$f=\sum_{\sigma\in\widehat{M}(\tau)}\int_0^\infty \mathcal Q_{\sigma,\lambda}^\tau f \diff \lambda,$$
 where  
 $\widehat{M}(\tau)$ is the set of  unitary irreducible representations $\sigma$ of $M=\SO(n-1)$   that occur in the restriction $\tau_{|M}$ and  $\mathcal Q^\tau_{\sigma,\lambda}$ is   the  generalized partial   spectral projection   defined for $f\in L^2(G,\tau)$   by
$$
\mathcal{Q}_{\sigma, \lambda}^\tau f(g)=\nu_\sigma(\lambda)\left(\mathcal{P}_{\sigma, \lambda}^\tau\left(\mathcal{F}_{\sigma, \lambda}^\tau f(\cdot)\right)(g), \right. \quad \text{ for a.e. $\lambda\in \mathfrak{a}^\ast$}.
$$
Above, $\mathcal F_{\sigma,\lambda}^\tau$ is the partial Helgason Fourier transform  (see \eqref{partial-fourier}), $\mathcal P_{\sigma,\lambda}^\tau$  is the   Poisson transform (see \eqref{poisson}) and $\nu_\sigma(\lambda)$ is the Plancherel density. Moreover,  $\mathcal{Q}_{\sigma, \lambda}^\tau f$ is (for a.e. $\lambda\in (0,\infty)$) a joint eigenform of the algebra of $G$-invariant differential operators acting on smooth sections of $G\times_K V_\tau$. 
Therefore the family of operators   
 $Q_\lambda^\tau: =\left(\mathcal Q_{\sigma,\lambda}^\tau; {\sigma\in\widehat{M}(\tau)}\right)$ gives a spectral decomposition of the space of $L^2$-forms of degree $p$.

 Our main result is a characterization of the $L^2$-range of the generalized spectral projection $\mathcal Q^\tau_\lambda$, see Theorem \ref{main-th-proj}.   
   In particular, if $\tau=\tau_p$ with $p$ generic (say $p\neq\frac{n-1}{2}, \frac{n}{2}$) we prove that the operator $\mathcal Q^{\tau_p}$ is an isomorphism from  the space $L^2(G,\tau_p)$   onto $\mathcal E_{p,\mathbb R_>}^2(G,\tau_p)\oplus \mathcal E^2_{p-1,\mathbb R_>}(G,\tau_p)$, where  
$\mathcal E_{p,\mathbb R_>}^2(G,\tau_p)$ (resp. $\mathcal E_{p-1,\mathbb R_>}^2(G,\tau_p)$)  is a  space of $V_{\tau_p}$-valued measurable functions $f$ on $\mathbb R_>\times G$, such that $f_\lambda = f(\lambda,\cdot)$ is for a.e. $\lambda$ a co-colosed $p$-eigenforms (resp. closed $(p-1)$-eigenforms ) of $\Delta$   and satisfying 
$$
 \sup_{R}\int_0^\infty \diff\lambda\,\frac{1}{R} \int_{B(R)}\diff (gK)\, \left\| f_\lambda(g)\right\|_{\tau}^2<\infty,
$$   
see \eqref{def+30-05}.\\ The   projection on $\mathcal E_{p,\mathbb R_>}^2(G,\tau_p)$ (resp. $\mathcal E_{p-1,\mathbb R_>}^2(G,\tau_p)$) being $\mathcal Q^{\tau_p}_{\sigma_p}$ (resp. $\mathcal Q^{\tau_p}_{\sigma_{p-1}}$).
 Furthermore, 
  for $\lambda\in\mathbb R\setminus\{0\}$, we prove that
 \begin{equation*}%\label{asym-Q}
\lim _{R \rightarrow \infty} \int_0^\infty \diff\lambda\,\frac{1}{R} \int_{B(R)}\diff (gK)\, \left\|\mathcal{Q}_{ \lambda}^\tau f(g)\right\|_{\tau}^2   = \frac{1}{\pi}   \|f\|_{L^2(G,\tau)}^2 ,
\end{equation*}
for any $L^2$-differential form $f\in L^2(G,\tau)$.  The operator $\mathcal Q^d_\lambda$ (resp. $\mathcal Q^\delta_\lambda$) of Strichartz can then be seen as our    $\mathcal Q^{\tau_p}_{\sigma_{p-1},\lambda}$ (resp. $\mathcal Q^{\tau_p}_{\sigma_{p},\lambda}$).
 This  provides  a positive answer to the conjecture of Strichartz on differential forms.

 To prove Theorem \ref{main-th-proj} we will first study the  uniform $L^2$-boundedness and an image characterization of the Poisson transform of $L^2$-differential forms on the boundary $\partial H^n(\mathbb{R}^n)$. 
 This is the purpose of
 the second main theorem of this paper, see Theorem \ref{main-th-Poisson}. 
 The proof of this theorem is based on a scattering formula for Poisson integrals, see    Theorem \ref{asymp-Poisson}. The key point of proving this latter theorem is to establish  an asymptotic expansion   for the  $\tau$-spherical functions associated to the unitary principal series, see Proposition \ref{key-formula}.

 One should note that the Poisson transform characterization for symmetric spaces of noncompact type is closely related to Helgason's conjecture (see \cite{H1}). This   subject has been  extensively  studied   by many authors, in particular for vector bundles (see e.g.  \cite{BBK2, BIO, BO, olbrich, Yang, van}).\\

  %%%%%%%%%%%%
  
Let us  outline  the content of this paper.  After general preliminaries (section 2) we will set up the Helgason Fourier transform on the  bundle of differential forms and prove a Helgason Fourier restriction theorem (section 3). We obtain then (section 4) as a consequence a uniform estimate of the Poisson transform. Once all materials are introduced we give the announcement of the main theorems (section 5). The delicate part of the paper (section 6) deals with the proof of the asymptotic expansion of $\tau$-spherical functions from which we get the asymptotic expansion of Poisson integrals. The obtained asymptotic properties are then used in section 7 to prove  Theorem \ref{main-th-Poisson}. Additionally, we establish, in Section 8, an inversion formula for the Poisson transfrom, see Theorem \ref{inversion}. Finally, we reserve Section 9 for the proof of Theorem      \ref{main-th-proj}.  
An appendix is added. It contains the proof of technical lemmas.

\section{Notations and Preliminaries}

 Let $H^{n}(\mathbb{R})$ be the $n$-dimensional real hyperbolic space ($n \geqslant 2$)
  viewed as the rank one symmetric space of the noncompact type $G/K$ where  $G=\SO_0(n,1)$ and $K=\SO(n)$.

 Let  $\mathbf {\mathfrak g}\simeq\mathfrak{so}(n,1)$ and $\mathbf {\mathfrak k}\simeq\mathfrak{so}(n)$ be the Lie algebras of $G$ and $K$ respectively  and write $\mathfrak g=\mathfrak k\oplus \mathfrak p$ for the Cartan decomposition of $\mathfrak g$.
 The tangent space $T_o(G / K) \simeq   \mathfrak{p}$ of $G / K=H^n(\mathbb{R})$ at the origin $o=e K$ will be identified with the vector space $\mathbb{R}^n$.

  There exists an element $H_0\in\mathfrak p$ such that 
%Let
%$$
%H_0=\left(\begin{array}{ccc}
%0 & 0 & 1 \\
%0 & 0_{n-1} & 0 \\
%1 & 0 & 0
%\end{array}\right) \in \mathfrak{p}
%$$
%Then 
$\mathfrak{a}=\mathbb{R} H_0$ is a Cartan subspace in $\mathfrak{p}$. Let  $A=\exp\mathfrak a=\{a_t= e^{tH_0},\; t\in\mathbb R\}$ be the corresponding analytic Lie subgroup of $G$.
 We define  $\alpha \in \mathfrak{a}^*$ by $\alpha\left( H_0\right)=1$.  
Then the positive restricted root subsystem is $\Sigma^{+}(\mathfrak{g}, \mathfrak{a})=\{\alpha\}$ and the corresponding positive Weyl chamber is $\mathfrak a^+=\{H\in\mathfrak a,\; \alpha(H)>0\}\simeq (0,\infty)$.   We will identify $\mathfrak a_{\mathbb C}^*$ and $\mathbb C$ via the map $\lambda \alpha \mapsto \lambda$. Then the half-sum of positive roots is $\rho=(n-1)\alpha/2 = (n-1)/2$.

Let $\mathfrak{n}=\mathfrak{g}_\alpha$ be the   positive root subspace, $N=\exp(\mathfrak n)$,  $A^+=\exp(\mathfrak{a}^+)$ and $\overline{A^+}=\exp(\overline{\mathfrak{a}^+})$.  The groupe $G$ has an Iwasawa decomposition $G=KAN$. Thus, each $g\in G$ can be uniquely written as 
$
g=k(g) e^{H(g)} n(g)$, where $k(g)\in K$, $H(g)\in \mathfrak{a}$ and   $n(g)\in N
$. 
Furthermore, according to the Cartan decomposition $G=K\overline{A^+}K$, any $g\in G$ decomposes as $g=k_1 e^{A^+(g)} k_2$, where $k_1, k_2 \in K$ and where $A^+(g)\in\overline{\mathfrak a^+}$  is uniquely determined by $g$.
 
 %Let $\diff g$ (resp.  $\diff k$) the normalized Haar measure on $G$ (resp. $K$).  
 We define the invariant measure on $G/K$ by %In the Cartan decomposition, the Haar measure on $G$ writes
  \begin{equation}\label{chg-cartan}
\int_{G/K}   \diff (gK) f(gK)  =\int_0^{\infty} \diff t(2 \sinh t)^{n-1} \int_K \diff k f\left(k a_t \right).
\end{equation}
for any integrable function $f$  on $G/K$. Here $K$ is equipped with its normalized Haar measure.\\
%Let $\tau_p$ be the standard representation of $SO(n)$ on $\extp^p\mathbb{C}^n$. Then  it is known that the representation $\tau_p$ is irreducible unless $n$ is even and $p=\frac{n}{2}$ in which case it decomposes as 
%$\tau_{\frac{n}{2}}=\tau_{\frac{n}{2}}^+\oplus \tau_{\frac{n}{2}}^-$ with the corresponding  decomposition of the representation space $\extp^{\frac{n}{2}} \mathbb{C}^n=\extp_{+}^{\frac{n}{2}} \mathbb{C}^n \oplus \extp_{-}^{\frac{n}{2}} \mathbb{C}^n$, where 
%$
%\extp^{\frac{n}{2}}_{\pm} \mathbb{C}^n=\{\alpha\in  \extp^{\frac{n}{2}} \mathbb{C}^n; \star \alpha=\mu_{\pm}\}
%$.
%Here $\star$ is  the Hodge operator and $\mu_\pm=\pm 1 \, \textit{if}\,\,   \frac{n}{2} \, \textit{is even}$ and $\mu_\pm=\pm i \, \textit{if}\,\,  \frac{n}{2}\, \, \textit{ is odd}$.

 %Since the Hodge operator  induces an equivalence $\tau_p \sim \tau_{n-p}$,  we   therefore restrict to $0 \leqslant p \leqslant n / 2$. 
%In this case we have $  \extp^{\frac{n}{2}} H^n(\mathbb{R})= \extp_+^{\frac{n}{2}} H^n(\mathbb{R})\oplus  \extp^{\frac{n}{2}}_- H^n(\mathbb{R})$, with $\extp_{\pm}^{\frac{n}{2}} H^n(\mathbb{R})=G\times_K V_{\tau_{\frac{n}{2}}^{\pm}}$.

\subsection{Differential forms on $H^n(\mathbb{R})$ and $\partial H^n(\mathbb R)$} 
%Let $T_c^\ast H^n(\mathbb{R})$ be the complexified cotangent bundle of $H^n(\mathbb{R})$. 
A differential $p$-form on $ H^n(\mathbb{R})$ is a section of the $p$-th exterior power of the complexified cotangent bundle of $H^n(\mathbb{R})$, $ \extp^p H^n(\mathbb{R}):= \extp^p T_c^\ast H^n(\mathbb{R})$.\\ 
 For $1\leqslant p\leqslant n,$ let $\tau_{p}$ be the standard representation of $K$ on $V_{\tau_p}=\extp^{p} \mathbb{C}^n$. Notice that $\tau_p$ is equivalent to the $p$-th exterior power   of  the coadjoint representation $Ad^\ast$ of $K$ on $\mathfrak{p}_c^\ast$. Let $G\times_K V_{\tau_p}$ be the the $G$-homogeneous vector bundle   associated to $\tau_p$. Since $T_{eK} H^n(\mathbb{R})\simeq \mathfrak{p}\simeq \mathbb{R}^n$, then as $G$-homogeneous vector bundles we have   $\extp^p H^n(\mathbb{R})=G\times_K V_{\tau_p}$. As usual we shall identify the space  $C^\infty( \extp^p H^n(\mathbb{R}))$ of smooth differential  $p$-forms on  $H^n(\mathbb{R})$ with the space $C^\infty(G,\tau_p)$  of  smooth functions   $f:G\to V_{\tau_p}$ which are right $K$-covariant, i.e., 
\begin{equation}\label{t-equiv16}
   	f(g k)=\tau\left(k^{-1}\right) f(g) \quad \text{for all $g\in G$,  $k\in K$}.
   \end{equation}
 Similarly, we identify the space $L^2(\extp^p H^n(\mathbb R))$ of square integrable $p$-forms with the space  $L^2(G , \tau)$   of $V_\tau$-valued functions on $G$ satisfying the transformation rule \eqref{t-equiv16} and 
$$
\parallel f\parallel_{L^2(G,\tau)}:=\left(\int_{G/K} \diff (gK) \parallel f(g)\parallel^2_\tau\right)^{\frac{1}{2}}<\infty.
$$
Here $\parallel \, \parallel_\tau$ denotes a norm on $V_\tau$ making $\tau$ unitary. Notice also that $\parallel f(gK)\parallel_\tau=\parallel f(g)\parallel_\tau$ is well defined for $f$ satisfying  \eqref{t-equiv16}. \\
  
 %As $T_{eM}(K/M)\simeq \mathfrak{k}/\mathfrak{m}\simeq \mathbb{R}^{n-1}$ and since  $\sigma_q$ is equivalent to the $q$-th power of the coadjoint representation of $M$ on $(\mathfrak{k}/\mathfrak{m})^\ast_c$ we have $\extp^q T^\ast_c(K/M)\simeq K\times_M V_{\sigma_q}$ 
%as $K$-homogeneous vector bundles on $\partial H^n(\mathbb{R})$. Thus $C^{-\omega}(K,\sigma)$ (resp. $L^2(K,\sigma_q)$) identifies with  $C^{-\omega}(\extp^q T^\ast_c \partial H^n(\mathbb{R}))$ the space of $q$-hyperforms (resp. $(L^2(\extp^q T^\ast_c \partial H^n(\mathbb{R}))$ the space of square integrable $q$-forms) on $\partial H^n(\mathbb{R})=K/M$. 
 %Recall that the representation  $\sigma_q$ is irreductible for $q\neq\frac{n-1}{2}$ and $\sigma_{\frac{n-1}{2}}=\sigma_{\frac{n-1}{2}}^+\oplus \sigma_{\frac{n-1}{2}}^-$. 
%For the sequel 
% (Throughout this  paper  we will)

 %%%%%%%%%%%%%%%%%%%%%%%
%\subsection{Fourier-Helgason transform}

%Below we review some known facts on   the Fourier-Helgason  transform on $p$-forms. To this end we shall need  the branching rules for $(K,M)=(\SO(n),\SO(n-1))$,  see e.g. \cite{Pedon, BS, IT}.\\
It is known that the representation $\tau_p$ is irreducible unless $n$ is even and $p=\frac{n}{2}$ in which case it decomposes as 
$\tau_{{n}/{2}}=\tau_{{n}/{2}}^+\oplus \tau_{{n}/{2}}^-$ with the corresponding  decomposition of the representation space $\extp^{\frac{n}{2}} \mathbb{C}^n=\extp_{+}^{\frac{n}{2}} \mathbb{C}^n \oplus \extp_{-}^{\frac{n}{2}} \mathbb{C}^n$, where 
$
\extp^{\frac{n}{2}}_{\pm} \mathbb{C}^n=\{\alpha\in  \extp^{\frac{n}{2}} \mathbb{C}^n; \star \alpha=\mu_{\pm}\alpha\}
$. In this case we have 
$
\extp^{\frac{n}{2}} H^n(\mathbb{R})=\extp_+^{\frac{n}{2}} H^n(\mathbb{R})\oplus \extp_-^{\frac{n}{2}} H^n(\mathbb{R})
$
with $\extp_\pm^p H^n(\mathbb{R})=G\times_K \extp_{\pm}^{\frac{n}{2}} \mathbb{C}^n$. 
Here $\star$ is  the Hodge operator and $\mu_\pm=\pm 1$  if $\frac{n}{2}$ is even and $\mu_\pm=\pm i$ if $\frac{n}{2}$  is odd. Notice that the Hodge operator induces the equivalence $\tau_p\sim \tau_{n-p}$, hence, hereafter we shall restrict our discussion to $0\leqslant p\leqslant \frac{n}{2}$.

To distinguish between   representations of $K$ and   representations of $M$, we will 
  henceforth use the Greek letter   $\sigma$ to denote  those representations of    $M$. This notation will help clearly differentiate the representations associated with these two groups throughout our discussion, ensuring that the analysis remains precise and unambiguous.
  
For $0\leqslant q\leqslant n-1$, let  $\sigma_q$ be the standard representation of $M$ on $V_{\sigma_q}=\extp^q(\mathbb Ce_2\oplus\cdots\oplus\mathbb C e_n)=\extp^q\mathbb C^{n-1}$, where $(e_j)_{j=1}^n$ is the natural basis of $\mathbb C^n$. Then
 the  branching rules for $(K,M)=(\SO(n),\SO(n-1))$ is given as follows (see e.g. \cite{Pedon, BS, IT}) :
%The branching rules for $(K,M)$ gives the following result
\begin{enumerate}
  \item[$(1)$] if $p< \frac{n-1}{2}$, then ${\tau_p}_{|M}=\sigma_{p-1}\oplus \sigma_{p}$;
  
  %% and the representation space is decomposed as
 % $$
%\extp^p \mathbb{C}^n=e_1 \wedge\left(\extp^{p-1} \mathbb{C}^{n-1}\right) \oplus \extp^p \mathbb{C}^{n-1}.
%$$
  \item[$(2)$] if $p=\frac{n-1}{2}$, then ${\tau_{p}}_{|M}=\sigma_{p-1}\oplus\sigma_{p}^+\oplus \sigma_{p}^-$;% and the representation space is decomposed as
 % $$
%\extp^{\frac{n-1}{2}}\mathbb C^n=e_1 \wedge (\extp\nolimits^{\frac{n-1}{2}-1} \mathbb{C}^{n-1}) \oplus \extp\nolimits_{+}^{\frac{n-1}{2}} \mathbb{C}^{n-1} \oplus \extp\nolimits_{-}^{\frac{n-1}{2}} \mathbb{C}^{n-1}.
%$$

\item[$(3)$] if $p=\frac{n}{2}$, then ${\tau_{\frac{n}{2}}^\pm}_{|M}=\sigma_{\frac{n}{2}}$

 %then 
%${\tau_{\frac{n}{2}}}_{|M}= 
%{\tau_{\frac{n}{2}}^+}_{|M}\widetilde{\oplus}\; %{\tau_{\frac{n}{2}}^-}_{|M}= 
%\sigma_{\frac{n}{2}-1} \widetilde{\oplus}\;  \sigma_{\frac{n}{2}}$. \\
%decomposition into irreducible equivalent factors
% and the representation space is decomposed as
%$$
%\extp^{\frac{n}{2}} \mathbb{C}^n
%=\extp\nolimits_{+}^{\frac{n}{2}} \mathbb{C}^n\, \widetilde{\oplus} \extp\nolimits_{-}^{\frac{n}{2}} \mathbb{C}^n
%=e_1 \wedge (\extp\nolimits^{\frac{n}{2}-1} \mathbb{C}^{n-1})\, \widetilde{\oplus} \extp\nolimits^{\frac{n}{2}} \mathbb{C}^{n-1}.
%$$
% Since ${\tau_{\frac{n}{2}}^+}_{|M}\sim {\tau_{\frac{n}{2}}^-}_{|M}$ and $\sigma_{\frac{n}{2}-1}\sim \sigma_{\frac{n}{2}}$ we have also ${\tau_{\frac{n}{2}}}_{|M}\sim 2{\tau_{\frac{n}{2}}^\pm}_{|M}\sim 2\sigma_{\frac{n}{2}-1}\sim 2\sigma_{\frac{n}{2}}$.
  \end{enumerate}
 
 We will refer to   {\it generic case} for  $1 \leqslant p \leqslant \frac{n}{2}$ with $p\neq \frac{n-1}{2}, \frac{n}{2}$ (we will not deal with the well-known case $p=0$,  see e.g. \cite{BS}, \cite{Ionescu}) and    {\it special cases} for the cases where $n$ is odd and $p=\frac{n-1}{2}$ or $n$ even and $p=\frac{n}{2}$.
%Let us set  
%  $$\Lambda=\left\{\tau_1, \cdots,   \tau_{\frac{n-1}{2}}, \tau^\pm_{\frac{n}{2}}\right\}.$$

For $\tau=\tau_1, \cdots,   \tau_{\frac{n-1}{2}}, \tau^\pm_{\frac{n}{2}}$, we denote by  $\widehat{M}(\tau)$ the set of representations in $\widehat{M}$ that occur (with multiplicity one) in the restriction of $\tau$ to $M$.   According to  the above branching rules we have, 
\begin{equation*}%\label{M(tau)}
   	\widehat M(\tau)=\begin{cases}
   	\widehat M(\tau_p)=\{\sigma_{p-1},\sigma_{p}\} & \text{ for $p<\frac{n-1}{2}$},\\
   	\widehat M(\tau_p)=\{\sigma_{p-1},\sigma_{p}^+,\sigma_p^-\} & \text{ for  $p=\frac{n-1}{2}$},\\
   	\widehat M(\tau_p^\pm)=\{\sigma_{p}\} & \text{ for  $p=\frac{n}{2}$.}\\
   %	\widehat M(\tau_n)=\{\sigma^+_{n-1},\sigma^-_{n-1}\} &\text{ for $K=\Spin(n)$ and $n$ odd}\\
   	 	%\widehat M(\tau_n^\pm)=\{\sigma_{n-1}\} & \text{ for $K=\Spin(n)$ and $n$ even}
   	\end{cases}	
   \end{equation*}

Let $\sigma\in \widehat{M}(\tau)$. Since $T_{eM} K/M=\mathfrak k/\mathfrak m\simeq \mathfrak a^\perp\simeq \mathbb R^{n-1}$ then   the space $\bigwedge^q \partial H^n(\mathbb R)$ of differential $q$-forms on $\partial H^n(\mathbb R)$ is canonically isomorphic to the homogenous bundle  $K\times_M V_{\sigma}=K\times_M \bigwedge^q \mathbb C^{n-1}$, with $q=p, p-1$.   The space $L^2(\bigwedge^q \partial H^n(\mathbb R))$   of $L^2$    differentiel $q$-forms   on $\partial H^n(\mathbb R)$ will be identified with
the space  $L^2(K,\sigma)$    of vector valued functions   $F:K\to V_\sigma$ satisfying the identity
\begin{equation}\label{aout8}
	F(km)=\sigma(m)^{-1}F(k),\;  \text{for all $k\in K,\ m\in M$}
\end{equation}
and such that
$$\parallel F\parallel_{L^2(K,\sigma)}=\left( \int_K \parallel F(k)\parallel^2\, {\rm d}k\right)^{\frac{1}{2}}<\infty.
$$ 
Similarly, we will also identify the space $C^{-\omega}(\bigwedge^q \partial H^n(\mathbb R))$ of $q$-hyperforms on $\partial H^n(\mathbb R)$ with the space $C^{-\omega}(K,\sigma)$ of vector-valued hyperfunctions $F:K\to V_\sigma$ satisfying the identity \eqref{aout8}.

\subsection{Fourier-Helgason transform}
Below we review some known facts on   the Fourier-Helgason  transform on differential forms  following mainly the notations in \cite{Pedon}.

  For a fixed $\tau=\tau_1,\ldots,\tau_{\frac{n-1}{2}}, \tau^\pm_{\frac{n}{2}}$ %$\Lambda$
   and a  given $\sigma\in\widehat{M}(\tau)$ we define the partial Helgason Fourier transform   of $f\in  C^\infty_c(G,\tau)$, to be the $V_\sigma$-valued function on $\mathfrak a^*_{\mathbb C}\times K=\mathbb C\times K$  given by (see \cite{Campo, Pedon}),
 \begin{equation*}%\label{partial-Fourier-T}
 	\mathcal F_\sigma^\tau f(\lambda)(k):=
 	\sqrt{d_{\tau,\sigma}}
 	\int_G \diff g\,  e^{(i\lambda -\rho)H(g^{-1}k)} P_\sigma\tau(\kappa(g^{-1}k))^{-1}
f(g), 	
 \end{equation*}
where $d_{\tau,\sigma}=\frac{d_\tau}{d_\sigma}=\frac{\dim \tau}{\dim \sigma}$ and  $P_\sigma$    the orthogonal  projection of $V_\tau$ onto its $\sigma$-isotypic component $ V_\tau(\sigma)=V_\sigma$. One can see that    $\mathcal F_{\sigma}^\tau f(\lambda)$ belongs to   $C^\infty(K,\sigma)$. 
  In the sequel we will set
\begin{equation}\label{partial-fourier}
\mathcal F_{\sigma,\lambda}^\tau f(k)=\mathcal F_{\sigma}^\tau f(\lambda)(k).
\end{equation}
The Fourier transform $\mathcal F^\tau f$ of  $f\in  C^\infty_c(G,\tau)$  is defined as the family 
\begin{equation}\label{Fourier-T}
\mathcal F^\tau f=(\lambda\mapsto \mathcal F^\tau_{\sigma,\lambda}f)_{\sigma\in\widehat{M}(\tau)}.
\end{equation}

%Recall from the abstract Plancherel theorem that the space $L^2(G,\tau_p)$ of differential $p$-forms decomposes as follows
%$$L^2(G,\tau_p)= L^2(G,\tau_p)_{\mathrm{cont}}\oplus L^2(G,\tau_p)_{\mathrm{disc}}, $$
%where $L^2(G,\tau_p)_{\mathrm{cont}}$ (resp. $L^2(G,\tau_p)_{\mathrm{disc}}$) is the continuous (resp. discrete) part of $L^2(G,\tau_p)$. We also mention that $L^2(G,\tau_p)_{\mathrm{disc}}$ is trivial except for $p=\frac{n}{2}$.
 \begin{theorem}[see {\cite[Theorems 6.19 -- 6.23, 2.27]{Pedon}}]\label{pedon}
 	$(1)$ Let  $\tau=\tau_1,\ldots,\tau_{\frac{n-1}{2}}$. Then for a differential form $f\in C^\infty_c(G,\tau)$ we have the following inversion and Plancherel formulas
 	\begin{equation}\label{inversion-generic}
 		f(g)= \sum_{\sigma\in\widehat{M}(\tau)}\sqrt{d_{\tau,\sigma}}
\int_0^\infty \nu_\sigma(\lambda)\diff \lambda\,\int_K \diff k\,e^{-(i\lambda+\rho)H(g^{-1}k)} \tau(\kappa(g^{-1}k)) \mathcal F_{\sigma,\lambda}^\tau f(k),
 	\end{equation}
 	\begin{equation}\label{Planch-generic}
 		\|f\|^2_{L^2(G,\tau)}=	\sum_{\sigma\in\widehat{M}(\tau)}
\int_0^\infty \nu_\sigma(\lambda)\diff \lambda\, \|\mathcal F_{\sigma,\lambda}^\tau f\|^2_{L^2(K,\sigma)}.
 	\end{equation}
 	Moreover, the Fourier transform $\mathcal F^\tau=(\mathcal F_\sigma^\tau)_{\sigma\in\widehat{M}(\tau)}$   extends to bijective isometry from $L^2(G,\tau)$ onto the space
$$\bigoplus_{\sigma\in\widehat{M}(\tau)} L^2(\mathbb R_+; L^2(K,\sigma),\nu_\sigma(\lambda)\diff\lambda).$$
$(2)$ Let $\tau=\tau_{\frac{n}{2}}^\pm$. Then for  $f^\pm\in C^\infty_c(G,\tau_p^\pm)$ we have the following inversion and Plancherel formulas
\begin{equation}\label{inversion-special2}
 	\begin{aligned}
	f^\pm(g)=  
\int_0^\infty & \nu_{\sigma_\frac{n}{2}}(\lambda)\diff \lambda\,
\int_K \diff k\,
 e^{-(i\lambda+\rho)H(g^{-1}k)} \tau^\pm_{p}(\kappa(g^{-1}k)) \mathcal F_\lambda^{^\pm} f^\pm(k)  \\
&    
 +2^{1-2n}n \int_K \diff k\,e^{-(\frac{1}{2}+\rho)H(g^{-1}k)}     \tau_p^\pm(\kappa(g^{-1}k))  
 \mathcal F_{-\frac{i}{2}}^{^\pm}f^\pm(k),
\end{aligned}
 \end{equation}
\begin{equation}\label{plach-form-special2}
\begin{gathered}
\left\|f^{ \pm}\right\|_{L^2(G,\tau_p^\pm)}^2=
 \int_0^{\infty}  \nu_{\sigma_\frac{n}{2}}(\lambda)\diff\lambda
 \left\|\mathcal{F}^\pm_{\lambda} f^{\pm}\right\|_{L^2\left(K,\,\sigma_{\frac{n}{2}}\right)}^2  +2^{1-2 n} n\left\langle
\mathcal{F}^\pm_{-\frac{i}{2}} f^{\pm},
\mathcal{F}^\pm_{\frac{i}{2}} f^{\pm}
\right\rangle_{L^2(K,\,\nu_{\sigma_\frac{n}{2}})}.
\end{gathered}
\end{equation}
Moreover, the Fourier transform $\mathcal F^\pm$ extends to a bijective isometry from $L^2(G,\tau_p^\pm)$ onto 
$$L^2(\mathbb R_+; L^2(K,\sigma_p), \nu_{\sigma_\frac{n}{2}}(\lambda)\diff \lambda)\oplus \mathcal H_{\pi^\pm},$$
where $\pi^\pm$ is the discrete series of $G$ with trivial infinitesimal character.
\end{theorem}

For  $1\leqslant p \leqslant \frac{n-1}{2}$, $q=p-1, p$, the Plancherel densities $\nu_q(\lambda):=\nu_{\sigma_q}(\lambda)$ in the above theorem (part (1)) are given by
\begin{equation}\label{expl-gen-21-03}
  \nu_q(\lambda)=d_{pq}^{-1}\frac{1}{2^{2n-3}\Gamma^2(\frac{n}{2})(\lambda^2+(\rho-q)^2)}
  \begin{cases}
	 \displaystyle \lambda^2 \prod_{k=1}^{\frac{n-1}{2}}\left(\lambda^2+k^2\right) &\text{when $n$ is odd}\\
	  \displaystyle   \lambda\tanh(\pi \lambda)    \prod_{k=1}^{\frac{n}{2}}(\lambda^2+(k-\frac{1}{2})^2)&\text{when $n$ is even}
\end{cases}  	
\end{equation}

For $p=\frac{n}{2}$, the Plancherel density $\nu_{\sigma_\frac{n}{2}}$ in the above theorem (part (2)) is given by
\begin{equation}\label{expl-n-sur-2}
\nu_{\sigma_\frac{n}{2}}(\lambda) =	\frac{1}{2^{2 n-3} \Gamma^2(\frac{n}{2})}
	 \lambda \tanh(\pi \lambda) \prod_{k=2}^{n / 2}\left[\lambda^2+\left(k-\frac{1}{2}\right)^2\right] .
\end{equation}

%where
%\begin{equation}
%	c(\lambda)%=c^{(\frac{n}{2}-1,\frac{n}{2}+1)}(\lambda)
%	=2^{n+1-i\lambda}\frac{\Gamma\left(\frac{n}{2}\right)\Gamma(i\lambda)}{\Gamma\left(\frac{i\lambda+n+1}{2}\right)\Gamma\left(\frac{i\lambda-1}{2}\right)}%=...
%\end{equation}

\begin{lemma} For  $\tau=\tau_1,\ldots,\tau_{\frac{n-1}{2}}, \tau^\pm_{\frac{n}{2}}$   and any $\sigma\in \widehat{M}(\tau)$, the following   estimate holds,\footnote{Here  $\phi\asymp \psi$ for two positive functions means that their quotient  $\frac{\phi}{\psi}$ stay uniformly bounded.}
 \begin{equation}\label{planch-esti}
	\nu_\sigma(\lambda) \asymp \lambda^2(|\lambda|+1)^{{n-1}}, \quad \lambda\in \mathbb{R}.  %\text{ if $n$  is odd }   %\\
	%\nu_q(\lambda) \asymp (|\lambda|+1)^{{n}} & \text{ if $n$  is even }  
 \end{equation}
\end{lemma}
  \begin{proof}
  	This is a direct consequence of \eqref{expl-gen-21-03} and \eqref{expl-n-sur-2}.
  \end{proof}
 
\section{The Helgason Fourier restriction theorem}

  Next, we are interested  in proving a Helgason Fourier restriction theorem of suitable sections of the bundle  $G\times_K V_\tau$ and fixed $\tau\in\{\tau_1,\ldots,\tau_{\frac{n-1}{2}}, \tau^\pm_{\frac{n}{2}}\}$.

Before stating the main theorem of this section  we will first define  
    the   Radon transform of $f \in C_{\mathrm{c}}^{\infty}(G , \tau)$ as the $V_{\tau}$-valued function (see e.g. \cite{BOS})
$$
\mathcal R^\tau f (g)=\mathcal R^\tau f(t,k)=e^{\rho t } \int_N f(ka_tn) \diff  n,\quad g=ka_tn\in G.
$$
 Then, for each $\sigma\in \widehat M(\tau)$ we define the   partial Radon transform of $f\in C_{\mathrm{c}}^{\infty}(G , \tau)$ by
$$\mathcal R^\tau _\sigma f(t,k)=\sqrt{d_{\tau,\sigma}} P_\sigma \mathcal R^\tau f(t,k).$$

For every real number $R>0$ set $B(R):=\{gK\in G/K, d(o,gK)<R\}$, where $d(\cdot,\cdot)$ is the hyperbolic distance on $G/K=H^n(\mathbb R)$. In view of the inequality (see \cite[page 476]{H}),
$$
d(eK,ka_tnK)\geqslant \mid t\mid \quad k\in K, n\in N,
$$
  we easily see that, if $f\in C_c^\infty(G,\tau)$ with $\operatorname{supp}f\subset B(R)$, then  
\begin{align}\label{RS}
\operatorname{supp}  \mathcal{R}^\tau_\sigma f\subset [-R,R]\times K.
\end{align}

Finally, note that the    partial Radon transform  is related to   the partial Helgason Fourier transform   through the following formula :
\begin{equation}\label{Radon-Fourier}
\mathcal F^\tau_{\sigma,\lambda} f (k)= \mathcal F^\tau_\sigma f (\lambda,k)=\mathcal F_\mathbb{R}\left[\mathcal R^\tau
_\sigma f(\cdot,k)\right](\lambda),\quad k\in K,
\end{equation}
 where $\mathcal{F}_\mathbb{R} \phi(\lambda)=\int_{\mathbb{R}}{\rm e}^{-i\lambda t}\phi(t){\rm d}t$  is the Euclidean Fourier transform on $\mathfrak a=\mathbb R$.

 We come now to the statement of the    uniform continuity  estimate  for restriction of the Helgason Fourier transform.% operator

\begin{theorem}[Helgason Fourier restriction theorem]\label{fourier} Let $\tau\in\{\tau_1,\ldots,\tau_{\frac{n-1}{2}}, \tau^\pm_{\frac{n}{2}}\}$  and $\sigma\in\widehat{M}(\tau)$.

$(1)$ Suppose $n$ even. There exists a positive constant $C$ such that for any $\lambda\in\mathbb R\setminus\{0\}$ and $R>1$ we have 
\begin{equation}\label{fourier-rest} 
\|\mathcal F_{\sigma,\lambda}^{\tau} f \|_{L^2(K,\sigma)}\leqslant C R^{1/2} \nu_\sigma(\lambda)^{-1/2}     \|f\|_{L^2(G,\tau)},
\end{equation}
for every $f\in L^2(G,\tau)$ supported in $B(R)$. \\
In the special case $p=\frac{n}{2}$, the above estimate must be understood  for   $f=f^\pm\in L^2(G,\tau^\pm_{\frac{n}{2}})_{\mathrm{cont}}$ supported in $B(R)$.

$(2)$ Suppose $n$  odd. 
There exists a positive constant $C$ such that for any $\lambda\in\mathbb R\setminus\{0\}$ and $R>0$ we have
\begin{equation}\label{fourier-rest1} 
\|\mathcal F_{\sigma,\lambda}^{\tau} f \|_{L^2(K,\sigma)}\leqslant C R^{1/2} \nu_\sigma(\lambda)^{-1/2}     \|f\|_{L^2(G,\tau)},
\end{equation}
for every $f\in L^2(G,\tau)$ supported in $B(R)$.
\end{theorem}
 
\begin{proof}
% Let $\tau=\tau_p$ with  $p\neq \frac{n}{2}$. %$\tau=\tau_1,\cdots,\tau^\pm_{\frac{n}{2}}$ with $\tau\neq\tau_{\frac{n-1}{2}}$. 

 It is sufficient to show \eqref{fourier-rest} (resp. \eqref{fourier-rest1})  for $f\in C_c^\infty(G,\tau)$. By the Plancherel formula \eqref{Planch-generic}- \eqref{plach-form-special2} we have
 \begin{equation}\label{cont}
\|f\|_{L^2\left(G, \tau\right)}^2=  \sum_{\sigma\in \widehat{M}(\tau)} \int_0^{\infty}   \nu_\sigma(\lambda) \diff \lambda\,  
   \|\mathcal{F}^{\tau}_{\sigma,\lambda}f \|_{L^2\left(K, \sigma\right)}^2.
%\|\mathcal{F}^{\tau_n}_{\sigma_{n-1}^+,\lambda}f \|_{L^2\left(K, \sigma_{n-1}^{+}\right)}^2 
%+\|\mathcal{F}^{\tau_n}_{\sigma_{n-1}^-,\lambda}f \|_{L^2\left(K, \sigma_{n-1}^{-}\right)}^2 
\end{equation} 
If $p=\frac{n}{2}$, the discrete term $$
2^{1-2 n} n\left\langle
\mathcal{F}^{\tau_p^\pm}_{\sigma_p,-\frac{i}{2}} f^{\pm},
\mathcal{F}^{\tau_p^\pm}_{\sigma_p,\frac{i}{2}} f^{\pm}
\right\rangle_{L^2(K,\sigma_p)}
$$ must be added to the right-hand side of \eqref{cont}.
%The discrete part appear   only when $\tau=\tau_{\frac{n}{2}}^\pm$,but in this case consider   $L^2(G,\tau)_{\text{cont}}$.
So,  we need only   to prove that there exists $C>0$ such that for any $R>1$ (resp. 
$R>0$)  and $f\in C_c^\infty(G,\tau)$ with $\operatorname{supp}\,f\subset B(R)$, we have  
 \begin{equation}\label{ven-15}
\nu_\sigma(\lambda)\parallel \mathcal F_{\sigma,\lambda}^\tau f\parallel^2_{L^2(K,\sigma)}\leqslant C R\int_0^\infty \nu_\sigma(\lambda) \diff \lambda\, \parallel \mathcal F_{\sigma,\lambda}^{\tau} f\parallel^2_{L^2(K,\sigma)}.  
\end{equation}

(1) Suppose $n$ even. 
Use \eqref{Radon-Fourier} to rewrite \eqref{ven-15} as 
\begin{align}\label{R1}
 \parallel  \nu_\sigma(\lambda)^{1/2}
 \mathcal F_\mathbb{R}[\mathcal R^{\tau}_{\sigma} f(t,\cdot)](\lambda)\parallel^2_{L^2(K,\sigma)}\leqslant C R\int_0^\infty \nu_\sigma(\lambda) \diff \lambda\, \parallel \mathcal F_\mathbb{R}[\mathcal R^{\tau}_{\sigma} f(t,\cdot)](\lambda)\parallel^2_{L^2(K,\sigma)}. 
\end{align} 
   Since  $\nu_\sigma(\lambda) \asymp \lambda^2(1+|\lambda|)^{n-1}$ (see \eqref{planch-esti}), we may replace $\nu_\sigma(\lambda)^\frac{1}{2}$ by $\lambda(1+|\lambda|)^\rho$ and    $(1+|\lambda|)^\rho$ by a  positive smooth function $\eta_\rho(\lambda)$ on $\mathbb R$, according to  \cite[Lemma 4]{Anker}.\\			
  Then  the  inequality  \eqref{R1} is equivalent to 
 \begin{align}\label{F2}
CR^{-1}\parallel \mathcal{F}_\mathbb{R}[T\ast \mathcal{R}^{\tau}_{\sigma} f(t,\cdot)](\lambda)\parallel^2_{L^2(K,\sigma)}\leqslant \int_\mathbb{R} \diff \lambda\, \parallel \mathcal{F}_\mathbb{R}[T\ast \mathcal{R}^{\tau}_{\sigma} f(t,\cdot)](\lambda)\parallel^2_{L^2(K,\sigma)},
\end{align}
where $T$ is the tempered  distribution on $\mathbb{R}$ such that  $\mathcal{F}_{\mathbb{R}}T(\lambda)=\lambda\eta_\rho(\lambda)$.
 Now, let us prove \eqref{F2}.   As $T$ has compact support, let $R_0>1$ such that  $\operatorname{supp}  T\subset [-R_0,R_0]$. By  \eqref{RS} we have
$$
\operatorname{supp}  T\ast \mathcal{R}^{\tau}_{\sigma} f(\cdot ,k)\subset [-(R_0+R),(R_0+R)],
$$ for all $k\in K$. Therefore by  the Cauchy-Schwarz inequality and the Euclidean Plancherel theorem we have 
$$\begin{aligned}
\parallel \mathcal{F}_e[T\ast \mathcal{R}^{\tau}_{\sigma} f(t,\cdot)](\lambda)\parallel^2_{L^2(K,\sigma)}
&\leqslant (R+R_0)\int_K   \diff k \int_\mathbb{R} \diff t\parallel T\ast \mathcal{R}^{\tau}_{\sigma} f(\cdot,k)(t)\parallel^2  \\
&= \left(1+\frac{R_0}{R}\right)R\int_K \diff k \int_\mathbb{R}\diff \lambda \parallel \mathcal{F}_\mathbb{R}[ T\ast \mathcal{R}^{\tau}_{\sigma} f(\cdot,k)](\lambda)\parallel^2 ,
\end{aligned}$$
thus, the result follows.

%The proof for  $\tau=\tau^\pm_{\frac{n}{2}}$ is similar to the above, replacing $L^2(G,\tau_p)$ by $L^2(G,\tau^\pm_{\frac{n}{2}})_{\mathrm{cont}}$, so we omit it. 

$(2)$  Suppose $n$ odd. 
%By the Plancherel formula \autoref{Plan-Formula-odd} we have
% \begin{equation*} 
%\|f\|_{L^2\left(G, \tau\right)}^2=  \sum_{\sigma\in\widehat{M}(\tau)} \int_0^{\infty}   \nu_\sigma(\lambda) \diff \lambda\,  
%   \|\mathcal{F}^{\tau}_{\sigma,\lambda}f \|_{L^2\left(K, \sigma\right)}^2.
%%%\|\mathcal{F}^{\tau_n}_{\sigma_{n-1}^+,\lambda}f \|_{L^2\left(K, \sigma_{n-1}^{+}\right)}^2 
%%%+\|\mathcal{F}^{\tau_n}_{\sigma_{n-1}^-,\lambda}f \|_{L^2\left(K, \sigma_{n-1}^{-}\right)}^2 
%\end{equation*} 
%It is therefore sufficient to show that there exists a positive constant $C$ such that for any $R>0$ and $f\in L^2(G,\tau)$ with  $\operatorname{supp}\,f\subset B(R)$, we have, \kk{for any $\sigma\in \widehat{M}(\tau)$}
%\begin{equation}\label{sam27}
%\parallel  \nu_\sigma(\lambda)^{1/2}\mathcal{F}^{\tau}_{\sigma,\lambda} f\parallel^2_{L^2(K,\sigma)}\leq 
% CR\int_0^{\infty}   \nu_\sigma(\lambda) \diff \lambda\,  
% \| \mathcal{F}^{\tau}_{\sigma,\lambda} f\|_{L^2(K,\sigma)}^2.
 %\end{equation}
 To show \eqref{ven-15}, let $f\in C^\infty_c(G,\tau)$ with $\operatorname{supp}\,f\subset B(R)$.  \ Since in this case $\nu(\lambda)^{1/2}$  is polynomial in $\lambda$,   it follows that
\begin{align*}\begin{split}
\parallel \nu_{\sigma}(\lambda)^{1/2} \mathcal{F}^{\tau}_{\sigma,\lambda} f(k)\parallel &= \parallel \nu_\sigma(\lambda)^{1/2} \mathcal{F}_\mathbb{R}\mathcal{R}^{\tau}_{\sigma} f(\cdot,k)(\lambda)\parallel\\
&=\parallel  \mathcal{F}_\mathbb{R}[\nu_\sigma\left(\frac{d}{dt}\right)^{1/2}( \mathcal{R}^{\tau}_{\sigma} f(\cdot,k)](\lambda)\parallel\\
&\leqslant \int_\mathbb{R}  \diff t\parallel [\nu_\sigma\left(\frac{d}{dt}\right)^{1/2}( \mathcal{R}^{\tau}_{\sigma} f(t,k)]\parallel .
\end{split}\end{align*}
 As $\nu_\sigma\left(\frac{d}{dt}\right)^{1/2}$ is a differential operator, we have  $\operatorname{supp}\,[\nu_\sigma\left(\frac{d}{dt}\right)^{1/2}  \mathcal{R}^{\tau}_{\sigma} f(\cdot,k)]\subset   [-R,R]$. Then the  Cauchy-Schwarz inequality and the Euclidean Plancherel formula imply
 $$\begin{aligned}
	\parallel \nu_\sigma\left(\lambda\right)^{1/2} \mathcal{F}^{\tau}_{\sigma,\lambda} f(k)\parallel^2 
	&\leqslant c_0 R\int_\mathbb{R} \diff \lambda\,
	\Bigl\| 
	\mathcal{F}_\mathbb{R}\Bigl[\nu_\sigma\left(\frac{d}{dt}\right)^{1/2}\mathcal{R}^{\tau}_{\sigma}  f(\cdot,k)\Bigr] 
	\Bigr\|^2  \quad c_0=1, 2\\
&=c_0R\int_{\mathbb R}\nu_\sigma(\lambda) \diff \lambda\, \|\mathcal F_{\sigma,\lambda}^{\tau} f(k)\|^2.
\end{aligned}$$  
Integrating both parts of the above inequality with respect to $k$,   the result follows from \eqref{Planch-generic}. %This completes the proof of Theorem \ref{fourier}. 
  %We have to prove the result for $p=\frac{n}{2}$, but we shall need to introduce $L^2_{\textit{cont}}(G,\tau)$ 

\end{proof}

 %{\color{red} à supprimer : Recall that for $n$ even, $\mathcal F^{\tau_n^\pm}=\mathcal F^{\tau_n^\pm}_{\sigma_{n-1}}$ and for $n$ odd $\mathcal F^{\tau_n}=\left(\mathcal F^{\tau_n}_{\sigma_{n-1}^+},\mathcal F^{\tau_n}_{\sigma_{n-1}^-}\right)$. In addition in this case we have
%$$  \|\mathcal F_\lambda^{\tau_n} f \|^2_{L^2(K,\sigma_{n-1}^+)\oplus L^2(K,\sigma_{n-1}^-)}=
%\|\mathcal{F}^{\tau_n}_{\sigma_{n-1}^+}(f)(\lambda)\|_{L^2\left(K, \sigma_{n-1}^{+}\right)}^2+ 
%\|\mathcal{F}^{\tau_n}_{\sigma_{n-1}^-}(f)(\lambda)\|_{L^2\left(K, \sigma_{n-1}^{-}\right)}^2.$$
%Then as a corollary of the above theorem we have}

%{\color{red} plus besoin du cor

%\begin{corollary}  
%$(1)$ Suppose $n$ even. There exists a positive constant $C$ such that for $\lambda\in\mathbb R\setminus\{0\}$ and $R>1$ we have
%\begin{equation}\label{fourier-rest-bis}
%\|\mathcal F_{\lambda}^{\tau_n^\pm} f \|_{L^2(K,\sigma_{n-1})}\leq C R^{1/2} \nu(\lambda)^{-1/2}     \|f\|_{L^2(G,\tau)},
%\end{equation}
%for every $f\in L^2(G,\tau_n^\pm)$ supported in $B(R)$.

%$(2)$ Suppose $n$ odd. There exists a positive constant $C$ such that for $\lambda\in\mathbb R\setminus\{0\}$ and $R>0$ we have
%\begin{equation}\label{fourier-rest1-bis}
 % \|\mathcal F_{\lambda}^{\tau_n} f \|_{L^2(K,\sigma_{n-1}^+)\oplus L^2(K,\sigma_{n-1}^-)}\leq C R^{1/2} \nu(\lambda)^{-1/2},    \|f\|_{L^2(G,\tau_n)},
 % \end{equation}
%for every $f\in L^2(G,\tau_n)$ supported in $B(R)$.
%\end{corollary}
 
 %}

\section{Poisson transforms on differential forms}
In this section we shall introduce the  Poisson transform  for differential forms on $\partial H^n(\mathbb{R})$.  
Let $\tau_p=\tau_1,\cdots,\tau_{\frac{n-1}{2}},   \tau_{\frac{n}{2}}$.
%$\tau:=\tau_p$ be the exterior powers of the standard representation of $K$ on $\extp^p \mathbb{C}^n$. 
%From the  decomposition 
%$$
%\omega=\omega'+e_1\wedge \omega", \quad \textit{with}\quad \omega'\in \Lambda^p\mathbb{C}^n\,\, \textit{and}\, \, \omega"\in \Lambda^{p-1}\mathbb{C}^{n-1},
%$$
%we see that $\extp^p\mathbb{C}^n =\Lambda^p\mathbb{C}^{n-1}\oplus e_1\Lambda\,\Lambda^{p-1}\mathbb{C}^{n-1}$
Recall from section 2 that  $
\tau_{p_{\mid M}}=\sigma_{p-1}\oplus \sigma_{p}$, with $\sigma_p$  irreducible unless $n$ is odd and $p=\frac{n-1}{2}$. 

We set $V_{\sigma_p}=\bigwedge^p\mathbb{C}^{n-1}$ and $ V_{\sigma_{p-1}}=e_1\wedge(\bigwedge^{p-1}\mathbb{C}^{n-1})$ and we denote by 
   $\iota_q^p:V_{\sigma_q}\to V_{\tau_p}$   the natural inclusion.
   
    Let $\sigma_{q,\lambda}$ be the representation of $P=MAN$ on $V_{q,\lambda}=V_{\sigma_q}$ given by 
 $$
 \sigma_{q,\lambda}(man)=\sigma_q(m)a^{\rho-i\lambda}.
$$

This defines a principal series representation on the Hilbert space $L^2(G,\sigma_{q,\lambda})$ consisting of $P$-covariant functions $F:G\to V_{\sigma_q}$ whose restriction to $K$ belong to  $L^2(K,\sigma_q)$. The group action being the left regular representation. We denote by $C^{-\omega}(G,\sigma_{q,\lambda})$ the space of its hyperfunction vectors. 
Then as $G$-homogenous vector bundles we have 
$$
C^{-\omega}(G,\sigma_{q,i(q-\rho)})=C^{-\omega}(\extp^q \partial H^n(\mathbb{R})).
$$ 
%\sbs{delate :\\where $C^{-\omega}(\extp^q \partial H^n(\mathbb{R}))$ is the space of differential forms with hyperfunction coefficients (hyperforms on $\partial H^n(\mathbb{R}))$). 
%}\\

By the Iwasawa decomposition, the restriction map from $G$ to $K$ gives a $K$-isomorphism from $L^2(G,\sigma_{q,\lambda})$ onto $L^2(K,\sigma_q)$. This yields to the compact picture with the group action given by
$$
\pi_{\sigma,\lambda}(g)F(k)={\rm e}^{(i\lambda-\rho)H(g^{-1}k)}F(\kappa(g^{-1}k).
$$ 
 The space $C^{-\omega}(K,\sigma_q)$ of its hyperfunctions sections is then   $K$-isomorphic to  $C^{-\omega}(G,\sigma_{q,\lambda})$.
 
 For $\lambda\in\mathbb C$, the Poisson transform  is the $G$-equivariant map 
$$
\mathcal{P}^{\tau_p}_{\sigma_q,\lambda}: C^{-\omega}(G,\sigma_{q,\lambda})\to C^\infty(G,\tau_p)
$$
 defined by 
 $$
 \mathcal{P}^{\tau_p}_{\sigma_q,\lambda}F(g):=\sqrt{d_{\tau_p,\sigma_q}}
 \int_K \diff k\,\tau_p(k) \iota^p_q   F(gk),
 $$
which in the compact picture 
$$
\mathcal{P}^{\tau_p}_{\sigma_q,\lambda}: C^{-\omega}(K,\sigma_q)\to C^\infty(G,\tau_p)
$$ 
is given by
\begin{align}\label{poisson}
\mathcal{P}^{\tau_p}_{\sigma_q,\lambda}F(g)= \sqrt{d_{\tau_p,\sigma_q}}
 \int_K \diff k\,{\rm e}^{-(i\lambda+\rho)H(g^{-1}k)}  \tau_p(\kappa(g^{-1}k)) \, \iota^p_q   F(k),
\end{align}
Notice that the Poisson transform  $\mathcal{P}^{\tau_p}_{\sigma_p,\lambda}$ is up to a constant the Poisson transform $\Phi_p^{\rho-i\lambda}$ investigated by Gaillard in \cite{Gaillard86}. 

%Moreover, accordingly to    \cite{Gaillard86}, the Poisson transform $\mathcal{P}_{\sigma_p,\lambda}^{\tau_p}$  is an isomorphism from the space  $C^{-\omega}(G,\sigma_{p,\lambda})$  onto the space of coclosed   $p$-eigenforms of the Laplacian, under conditions on $\lambda$  ( see also Juhl in the case $n$ is even).

To state Gaillard's result let us review the description the algebra    $\mathbb{D}(G,  \tau_p)$   of $G$-invariant differential operators on $C^\infty(G,\tau_p)$.

%of the Hodge Laplacian on $\mathbb{H}^n$. %or more generally the algebra  the algebra $\mathbb D(G,\tau)$   of $G$-invariant differential operators acting on the space $C^\infty(G,\tau_p)$ of smooth differential $p$-forms.
 Let $d \colon C^\infty (\extp^k H^n(\mathbb R))  \to C^\infty (\extp^{k+1} H^n(\mathbb R))$
be the exterior differentiation operator. Use the Hodge-star operator 
 $\star : \extp^k H^n(\mathbb R) \to \extp^{n-k} H^n(\mathbb R)$ 
%
%$$
%\star:\extp^p T_c^\ast H^n(\mathbb{R})\to \extp^{n-p} T_c^\ast H^n(\mathbb{R})
%$$ 
to define the co-differentiation operator $  \delta  =(-1)^{n(k+1)+1} \star d\, \star :  C^\infty(\extp^k H^n(\mathbb R))\to C^\infty\extp^{k-1} H^n(\mathbb R))$.
 Then the Hodge-de Rham Laplacian is defined on $C^\infty (\extp H^n(\mathbb R))$ by
$\Delta =d \delta + \delta d$.

 In the sequel, one should work with the equivalent operators   of $\star$, $d$, $ \delta $ and $\Delta$ according to  the identification $C^\infty(\extp^pH^n(\mathbb R))\simeq C^\infty(G,\tau_p)$, but we continue using the notation $\star$, $d$, $ \delta $ and $\Delta$.
 
 %With the introduction of these operators, we take the opportunity to describe the algebra    $\mathbb{D}(G,  \tau)$   of $G$-invariant differential operators on $C^\infty(G,\tau)$. 
 As stated in  {\cite[Theorem 3.1]{Gaillard3}}, see also {\cite[Corollary 2.2]{Pedon}}, we have
$$
 \mathbb{D}(G,  \tau) = \begin{cases}\mathbb{C}\left[ d \delta ,  \delta d\right] & \text { if } \tau=\tau_p, \, p<\frac{n-1}{2}\\ 
\mathbb{C}\left[d \delta , \star d \right] & \text { if } \tau=\tau_p,\; p=\frac{n-1}{2}, \\ 
\mathbb{C}[\Delta] & \text { if } \tau=\tau_p^\pm,\;  p=\frac{n}{2},\\
\mathbb{C}[\star, d\star d] & \text { if } \tau=\tau_p,\;  p=\frac{n}{2},
\end{cases}
$$
In particular,  $\mathbb{D}\left(G,  \tau\right)$ is a commutative algebra except for $\tau=\tau_{n/2}$ (last case). \\

For $1\leqslant p\leqslant \frac{n}{2}$  and  $\lambda\in \mathbb{C}$, we consider the following space of coclosed $p$-eigenforms of the Hodge-de Rham Laplacian:
$$
\mathcal{E}_{\sigma_p,\lambda}(G, \tau_p)=
	\{f\in C^\infty(G, \tau_p) : \; \Delta\, f=(\lambda^2+(\rho-p)^2)f, \,\text{ and }\;  \delta f=0\} 
$$
% be the space of coclosed   $p$-eigenforms of the Laplacian. 
Then Gaillard's result may be stated as follows:
 
 \begin{theorem}[{\cite[Theorem 2']{Gaillard86}}]\label{gaillard} 
Let $\tau_p=\tau_1,\cdots,\tau_{\frac{n-1}{2}}, \tau_{\frac{n}{2}}$  and 
$\lambda\in \mathbb{C}$. 
The Poisson transform 
$$
\mathcal{P}^{\tau_p}_{\sigma_p,\lambda} \colon C^{-\omega}(K,\sigma_p) \to \mathcal{E}_{\sigma_{p},\lambda}(G, \tau_p)
$$
is a  $K$-isomorphism  if and only if 
\begin{equation}\label{Pbijectivity}
i\lambda\notin \mathbb{Z}_{\leqslant 0}-\rho \cup \{p-\rho\}.
\end{equation}
\end{theorem}

 A similar result was announced in \cite[Theorem 4]{Juhl}.

Next, to handle $\mathcal{P}^{\tau_p}_{\sigma,\lambda}$ for all $\sigma$ arising  from the restricting $\tau_p$ to $M$, we need  a result analogous to Theorem \ref{gaillard}. To this end we introduce  the following common eigenspaces:

  For $1\leqslant p< \frac{n-1}{2}$,  
$$
\mathcal{E}_{\sigma_{p-1},\lambda}(G, \tau_p)
=\{f\in C^\infty(G, \tau_p) : \;  \Delta f= (\lambda^2+(\rho-p+1)^2)f \; \text{and}\; d f=0\}.
$$

 For $p=\frac{n-1}{2}$,  
$$
\mathcal{E}_{\sigma_{p-1},\lambda}(G, \tau_p)=\{f\in C^\infty(G, \tau_p) : \; \Delta\, f=(\lambda^2+1)f, \,\text{ and }\; d f=0\},  $$
and 
$$
\mathcal{E}_{\sigma^\pm_p,\lambda}(G, \tau_p)
=\{f\in C^\infty(G, \tau_p) : \;  \star df=\pm i^{p^2-1} \lambda f  \; \text{and}\;  \delta  f=0\}.
$$

\begin{corollary}\label{bijective1}
%Let %$1\leqslant p\leq n$ and 
Let $\lambda\in \mathbb{C}$. 

$(1)$ For $p<\frac{n-1}{2}$,  the Poisson transform 
$$
\mathcal{P}^{\tau_p}_{\sigma_{p-1},\lambda}:C^{-\omega}(K,\sigma_{p-1})\to \mathcal{E}_{\sigma_{p-1},\lambda}(G,\tau_p)
$$
is a  $K$-isomorphism
if and only if 
\begin{equation}
i\lambda\notin \mathbb{Z}_{\leqslant 0}-\rho\cup \{\rho-p+1\}.	
\end{equation}

$(2)$ For $p=\frac{n-1}{2}$,  the Poisson transform 
$$
\mathcal{P}^{\tau_p}_{\sigma_{p}^\pm,\lambda}:C^{-\omega}(K,\sigma_{p}^\pm)\to \mathcal{E}_{\sigma_{p}^\pm,\lambda}(G,\tau_p)
$$
is a  $K$-isomorphism 
if and only if $$i\lambda\notin \mathbb Z_{\leqslant 0}-\rho.$$
\end{corollary}
 \begin{proof}
	Notice that the Hodge operator induces the equivalences $\tau_p\sim \tau_{n-p}$ and $\sigma_q\sim \sigma_{n-1-q}$. 
	
	(1) By using the identity (see \cite{Pedon})
$$
\star (\mathcal{P}_{p-1,\lambda}^{\tau_{p}}f)=(-1)^{p(p-1)}\mathcal{P}_{n-p,\lambda}^{\tau_{n-p}}(\star  f),
$$
 as well as the relations
$$ \begin{aligned}
 \delta\star&=(-1)^{n+1-p^2}\star\,d\\
   \Delta\star&=\star\Delta,
\end{aligned}$$
we easily see that the following diagram  
$$
\xymatrix{
    C^{-\omega}(K,\sigma_{p-1}) \ar[r]^{\mathcal{P}^{\tau_p}_{\sigma_{p-1},\lambda}} \ar[d]_\star  & \mathcal{E}_{\sigma_{p-1},\lambda}(G,\tau_p) \ar[d]^\star \\
    C^{-\omega}(K,\sigma_{n-p}) \ar[r]^{\mathcal{P}^{\tau_{n-p}}_{\sigma_{n-p},\lambda}} & \mathcal{E}_{\sigma_{n-p},\lambda}(G,\tau_{n-p})
  }
  $$
is commutative and the first part of the  corollary follows from Theorem \ref{gaillard}.\\
(2)  Let  $p=\frac{n-1}{2}$. Then $\sigma_p=\sigma^+_p\oplus \sigma^-_p$.  By using    \cite[(4.24) and (4.38)]{Pedon} we have
$$\operatorname{Im} \mathcal{P}_{\sigma^\pm_\frac{n-1}{2},\lambda}^{\tau_\frac{n-1}{2}}\subset  \mathcal{E}_{\sigma_{\frac{n-1}{2}}^\pm,\lambda}(G, \tau_{\frac{n-1}{2}}).
$$ 
Since $(\star d)^2=(-1)^{\frac{n(n-1)}{2}}\delta d$, it follows that $\mathcal{E}_{\sigma_p,\lambda}(G,\tau_p)=\mathcal{E}_{\sigma_p^+,\lambda}(G,\tau_p)\oplus \mathcal{E}_{\sigma_p^-,\lambda}(G,\tau_p)$. This together with Theorem \ref{gaillard} give the desired  result.
\end{proof}

 It remains to prove the analogue of Theorem \ref{gaillard} for $p=\frac{n}{2}$ and the pair $(\tau_p^\pm,\sigma_p)$. 
As we were unable to derive this result directly from Gaillard's theorem, we will instead rely on a more general result established in \cite{olbrich} by Olbrich.
In fact, we could have applied Olbrich's  result to the other cases, but we intentionally chose to remain within the framework of differential forms.

Let
$$
\mathcal{E}_{\sigma_p, \lambda}\left(G, \tau_p^{ \pm}\right)=\left\{f^{ \pm} \in C^{\infty}\left(G, \tau_p^{ \pm}\right): \Delta f^{ \pm}=\left(\lambda^2+\frac{1}{4}\right) f^{ \pm}\right\},\;\; p=\frac{n}{2}.
$$
Then the corollary below follows from \cite[Theorem 4.16]{olbrich}.
\begin{corollary}\label{cor-olbrich}
For $p=\frac{n}{2}$, the Poisson transform
$$
\mathcal{P}_{\sigma_p, \lambda}^{\tau_p^{ \pm}}: C^{-\omega}\left(K, \sigma_p\right) \rightarrow \mathcal{E}_{\sigma_p, \lambda}\left(G, \tau_p^{ \pm}\right)
$$
is a  $K$-isomorphism  if and only if
$$
i \lambda \notin \mathbb{Z}_{\leqslant 0}-\rho \cup\left\{-\frac{1}{2}\right\} .
$$	
\end{corollary}

 Below, we will begin discussing the precise image of square-integrable forms on $\partial H^n(\mathbb R)$ under  the Poisson transfrom. This image is expected to constitute a subspace of the eigenspaces we introduced earlier.
  
 As a consequence of Theorem \ref{fourier}, we have the following uniform estimate for the Poisson transform.
 
\begin{proposition}\label{pro-est-Poisson}
 Let $\tau\in\Lambda$ and $\sigma\in\widehat M(\tau)$.  

(1) Suppose $n$ even. Then there exists a positive constant $C$ such that for $\lambda\in\mathbb R\setminus\{0\}$ we have
\begin{equation*}%\label{esti-Poisson-enen}
\sup_{R>1} \frac{1}{R}\int_{B(R)} \diff (gK)\,\|\mathcal P_{\sigma,\lambda}^\tau F(g)\|_{\tau}^2   \leqslant C  \nu_\sigma(\lambda)^{-1} \|F\|^2_{L^2(K,\sigma)},
\end{equation*}
for any $F\in L^2(K,\sigma)$.

(2) Suppose $n$ odd. Then there exists a positive constant $C$ such that for $\lambda\in\mathbb R\setminus\{0\}$ we have
\begin{equation*}%\label{esti-Poisson-odd}
\sup_{R>0}  \frac{1}{R}\int_{B(R)}\diff (gK)\, \|\mathcal P_{\sigma,\lambda}^\tau F(g)\|_{\tau}^2   \leqslant C  \nu_\sigma(\lambda)^{-1} \|F\|^2_{L^2(K,\sigma)},
\end{equation*}
for any $F\in L^2(K,\sigma)$.
\end{proposition}

 \begin{proof}
Since $\lambda$ is real and $\tau$ is unitary, we have
$$\int_{B(R)} \diff (gK)\,\langle \mathcal P^\tau_{\sigma,\lambda} F(g), f(g\rangle_\tau  
=\int_K \diff k\,\langle F(k),\mathcal F_{\sigma,\lambda}^\tau f(k)\rangle_\sigma,$$
for any $F\in L^2(K,\sigma)$ and $f\in L^2(G,\tau)$ with $\operatorname{supp}f\subset B(R)$. Hence
$$\begin{aligned}
\left|
\int_{B(R)} \diff (gK)\,\langle \mathcal P^\tau_{\sigma,\lambda} F(g), f(g\rangle_\tau  
\right|
&\leqslant \|F\|_{L^2(K,\sigma)} \left(\int_K \diff k\,\| \mathcal F_{\sigma,\lambda}^\tau f(k)\|^2_\sigma \right)^{1/2}\!\!,
\end{aligned}
$$
and by the  Helgason Fourier restriction theorem (Theorem \ref{fourier}) we get
$$\begin{aligned}
\left|
\int_{B(R)}  \diff (gK)\,\langle \mathcal P^\tau_{\sigma,\lambda} F(g), f(g\rangle_\tau 
\right|
&\leqslant C  \nu_\sigma(\lambda)^{-1/2} R^{1/2}   \|F\|_{L^2(K,\sigma)}  \|f\|_{L^2(G,\tau)}.
\end{aligned}
$$
The desired estimate follows  by taking the supremum over all $f$ with $ \|f\|_{L^2(G,\tau)}=1$.
\end{proof}

 To  obtain  more accurate  understanding on the $L^2$-range of $\mathcal P_{\sigma,\lambda}^\tau$, Proposition \ref{pro-est-Poisson}  motivates us to introduce the eigenspace
 %$\mathcal E_{\sigma,\lambda}^2(G,\tau)$ % to be the space defined by %the following spaces  
$$\begin{aligned}
\mathcal E_{\sigma,\lambda}^2(G,\tau )&=\mathcal E_{\sigma,\lambda}(G,\tau)\cap B^*(G,\tau),
% \mathcal E_{\sigma_{n-1}^\pm,\lambda}^2(G,\tau_n )&=\mathcal E_{\sigma_{n-1}^\pm,\lambda}(G,\tau_n )\cap B^*(G,\tau_n ) \quad \text{if $n$ is odd}.
 \end{aligned}$$
 where   $B^*(G,\tau)$  is the space of  $f\in L^2_{\text{loc}}(G,\tau)$ such that
 \begin{equation*}%\label{norm*}
 \|f \|^2_*=\begin{cases}
 \displaystyle\sup_{R>1} \frac{1}{R}\int_{B(R)}\diff (gK)\,\|f(g)\|^2_\tau  <\infty &  \text{if $n$ is even},\\
   \displaystyle\sup_{R>0} \frac{1}{R}\int_{B(R)}\diff (gK)\,\|f(g)\|^2_\tau  <\infty &  \text{if $n$ is odd}.
 \end{cases}
\end{equation*}
According to    Theorem \ref{gaillard}, Corollary \ref{bijective1} and Corollary \ref{cor-olbrich}, for $\lambda\in\mathbb R\setminus\{0\}$, the  Poisson transform  
\begin{equation}\label{poi-30-mars}
	\mathcal P_{\sigma,\lambda}^{\tau} \colon L^2(K,\sigma) \to \mathcal E_{\sigma,\lambda}^2(G,\tau )
\end{equation}
 %\begin{align} 
 %\mathcal P_{\sigma_q,\lambda,}^{\tau_p} &\colon L^2(K,\sigma_q) \to \mathcal E_\lambda^2(G,\tau_p ) & \text{for generic $p$ and $q=p-1,p$ }	\label{p1}\\
 %\mathcal P_{\sigma_p^\pm,\lambda,}^{\tau_p} &\colon L^2(K,\sigma_{p}^\pm) \to \mathcal E_\lambda^2(G,\tau_{p} ) &\text{for  $p=\frac{n-1}{2}$ } \label{p2}\\
 %\mathcal P_{\sigma_{p-1},\lambda,}^{\tau_p} &\colon L^2(K,\sigma_{p-1}) \to \mathcal E_\lambda^2(G,\tau_{p} ) & \text{for  $p=\frac{n-1}{2}$ } \label{p3}\\
%\mathcal P_{\sigma_p,\lambda,}^{\tau_p^\pm} &\colon L^2(K,\sigma_{p}) \to \mathcal E_\lambda^2(G,\tau_{p}^\pm ) & \text{for  $p=\frac{n}{2}$ } \label{p4}
 %\end{align}
is  a continuous injective linear map for any $\tau\in\{\tau_1,\cdots,\tau_{\frac{n-1}{2}},\tau^\pm_{\frac{n}{2}} \}$   and $\sigma\in \widehat{M}(\tau)$.
We will show below, that   \eqref{poi-30-mars}   is in fact  a bijective map. 
 
\section{Announcement of the main theorems}

Let $\tau=\tau_1,\cdots,\tau_{\frac{n-1}{2}},\tau^\pm_{\frac{n}{2}} $  and $\sigma\in\widehat M(\tau)$.  For $f \in L^2(G,\tau)$ 
%(resp. $f \in L^2(G,\tau)_{\mathrm{cont}}$ for $p=\frac{n}{2}$) 
we define the function 
 $\mathcal{Q}^\tau_{\sigma,\lambda} f \in C^{\infty}(G,\tau)$ for a.e., $\lambda \in (0,+\infty)$ by
  \begin{equation*}%\label{proj-Q1}
 \mathcal{Q}^\tau_{\sigma,\lambda} f(g)
 =\sqrt{d_{\tau,\sigma}}{\nu_\sigma(\lambda)} \int_K \diff k \,
  e^{-(i\lambda+\rho)H(g^{-1}k)} \tau(\kappa(g^{-1}k))^{-1}
 \mathcal{F}_{\sigma,\lambda}^\tau f(k). \label{19m1}
%& = {\nu_\sigma(\lambda)}  \mathcal P_{\sigma,\lambda}^\tau \left(\mathcal{F}_{\sigma,\lambda}^\tau f\right)(g) &\label{19m2}
 \end{equation*}
 For $\tau=\tau^\pm_{n/2}$ we would need   to consider  $\mathcal Q_{\sigma,\lambda}^\tau$ defined on $L^2(G,\tau)_{\operatorname{cont}}$.\\ 
  Notice that by the Plancherel Theorem (Theorem \ref{pedon})  the operator $\mathcal Q_{\sigma,\lambda}^\tau$ is well defined  and we have the following spectral decomposition  
 $$f=   \sum_{\sigma\in\widehat{M}(\tau)}\int_0^\infty    \diff \lambda\, \mathcal Q_{\sigma, \lambda}^\tau f.  $$
Next, we define the operator $\mathcal{Q}^\tau_{\lambda}  =\left(\mathcal{Q}^\tau_{\sigma,\lambda} \right)_{\sigma\in\widehat{M}(\tau)} : L^2(G,\tau)\to \oplus_{\sigma\in\widehat{M}(\tau)}\mathcal E_{\sigma,\lambda}(G,\tau)$. 
  The family    $\left(\mathcal{Q}^\tau_{\lambda}\right)_{\lambda \in \mathbb R_>}$ are called   here generalized spectral projections.   
  Write 
$$\mathcal Q^\tau f(\lambda,g)=\mathcal Q^\tau_\lambda f(g) \; \text{for a.e. } (\lambda,g)\in (0,+\infty)\times G.$$
To investigate an image characterization of $Q^\tau$,
let us  introduce the space 
\begin{equation}\label{def+30-05}
	\mathcal E^2_{\mathbb R_>}(G,\tau):=\oplus_{\sigma\in\widehat M(\tau)} \mathcal E^2_{\sigma,\mathbb R_>}(G,\tau)
\end{equation}
where $\mathcal E^2_{\sigma,\mathbb R_>}(G,\tau)$ is   the space of $V_\tau$-valued measurable functions   $\psi $ on $\mathbb R_>\times G$ such that $\psi(\lambda,\cdot)\in \mathcal E_{\sigma,\lambda}(G,\tau)$   for a.e. $\lambda\in(0,+\infty)$ and $\| \psi\|_+<\infty$, with
\begin{equation*}%\label{norm+}
    \| \psi\|_{+}^2 :=
\begin{cases}
\displaystyle\sup_{R>1} \int_0^\infty \diff\lambda\,\frac{1}{R}\int_{B(R)}\diff (gK)\, \|\psi(\lambda,g)\|_\tau^2  <\infty &  \text{if $n$ is even},\\
\displaystyle\sup_{R>0} \int_0^\infty \diff\lambda\,\frac{1}{R}\int_{B(R)} \diff (gK)\,\|\psi(\lambda,g)\|_\tau^2 
   <\infty &  \text{if $n$ is odd}.
\end{cases}
\end{equation*}
We also set  
\begin{equation}\label{norm-1-juin}
\|\phi \|_+^2 = \sum_{\sigma\in\widehat M(\tau)} \|\phi_\sigma\|_+^2,
\end{equation}
for   $\phi= \left( \phi_\sigma\right)_{\sigma\in\widehat{M}(\tau)}$.
 
Now, we state our   main theorems.
\begin{theorem}[Strichartz's conjecture for the spectral projections]\label{main-th-proj}% 
  Let $\tau=\tau_1,\cdots,\tau_{\frac{n-1}{2}}, \tau^\pm_{\frac{n}{2}}$. 
  %\sbs{ delate : ($p\neq \frac{n}{2}$)}% such that $\tau\neq \tau^\pm_{\frac{n}{2}}$ and $\sigma\in \widehat{M}(\tau)$.

$(1)$  
 There exists a positive constant $C$,   such that  for any $f\in L^2(G,\tau)$ 
 %(resp. $f\in L^2(G,\tau_p^\pm)_{\mathrm{cont}}$ when $p=\frac{n}{2}$) 
we have
\begin{equation}\label{esti-main1}
C^{-1} \|f\|_{L^2(G,\tau )} \leqslant\left\|\mathcal{Q}^\tau f\right\|_{+} \leqslant C \|f\|_{L^2(G,\tau)},
\end{equation}
  Furthermore, we have
\begin{equation}\label{asym-Q}
\lim _{R \rightarrow \infty} \int_0^\infty \diff\lambda\,\frac{1}{R} \int_{B(R)}\diff (gK)\, \left\|\mathcal{Q}_{ \lambda}^\tau f(g)\right\|_{\tau}^2   = \frac{1}{\pi} \|f\|_{L^2(G,\tau)}^2 .
\end{equation}
 
 %$\gamma_0= 2 |c_\sigma(\lambda,\tau)|^2\nu_\sigma(\lambda,\tau)$ ça ne dépends pas de $\lambda$ . 
  %2/\pi$ for $n$ odd and $\gamma_0=4/\pi$ for $n$ even.
 For $\tau= \tau_{\frac{n}{2}}^\pm$, one have to consider $f$ in the continuous part of $L^2(G,\tau)$.

$(2)$  
For  $\tau\neq \tau_{\frac{n}{2}}^\pm$, 
the linear map $\mathcal{Q}^\tau$ is a topological isomorphism from $L^2(G,\tau)$ onto $\mathcal{E}_{\mathbb R_>}^2(G,\tau)$.

 $(3)$ For $\tau= \tau_{\frac{n}{2}}^\pm$, the linear map $\mathcal{Q}^\tau$ is a topological isomorphism from $L^2(G,\tau)_{\mathrm{cont}}$ onto $\mathcal{E}_{\mathbb R_>}^2(G,\tau)$.

%$(\mathrm{II})$ Suppose $\tau=\tau^\pm_{\frac{n}{2}}$ with $n$ even and $\sigma=\sigma_{\frac{n}{2}}$.

%$(1)$  
% There exists a positive constant $C$   such that  for any $f\in L^2(G,\tau)_{\mathrm{cont}}$ we have
%\begin{equation}\label{esti-main1}
%C^{-1} \|f\|_{L^2(G,\tau )} \leqslant\left\|\mathcal{Q}^\tau f\right\|_{+} \leqslant C \|f\|_{L^2(G,\tau)} .
%\end{equation}
%  Furthermore, we have
%\begin{equation}\label{asym-Q}
%\lim _{R \rightarrow \infty} \int_0^\infty \diff\lambda\,\frac{1}{R} \int_{B(R)}\diff (gK)\, \left\|\mathcal{Q}_{ \lambda}^\tau f(g)\right\|_{\tau}^2   = \gamma_0 \|f\|_{L^2(G,\tau)}^2 .
%\end{equation}
% where $\gamma_0$ is a constant depending on $n$.
%%%where  $\gamma_0= 2|c_\sigma(\lambda,\tau)|^2\nu_\sigma(\lambda,\tau)$??. 
%2/\pi$ for $n$ odd and $\gamma_0=4/\pi$ for $n$ even.

%$(2)$  The linear map $\mathcal{Q}^\tau$ is a topological isomorphism from $L^2_{\mathrm{cont}}(G,\tau)$ onto $\mathcal{E}_{\mathbb R_+}^2(G,\tau)$.

   \end{theorem}

%When $\tau=\tau_{\frac{n}{2}}^\pm$ the theorem involves   $L^2_{\mathrm{cont}}(G,\tau)$ insteadof $L^2(G,\tau)$.
 
Since the spectral projection  $\mathcal{Q}_{\sigma, \lambda}^\tau f$  can be written as 
$$
\mathcal{Q}_{\sigma, \lambda}^\tau f(g)=   {\nu_\sigma(\lambda)}   \mathcal{P}_{\sigma, \lambda}^\tau\left(\mathcal{F}_{\sigma, \lambda}^\tau f\right)(g)
\quad \text{for a.e. } \lambda \in \mathbb{R}_>,$$
the above theorem will be a   consequence of the following theorem.

\begin{theorem}[Strichartz's conjecture for the Poisson transform] \label{main-th-Poisson} 
Let $\tau=\tau_1,\cdots,\tau_{\frac{n-1}{2}}, \tau^\pm_{\frac{n}{2}}$,
  $\sigma\in \widehat{M}(\tau)$ and $\lambda\in\mathbb R\setminus\{0\}$.

$(1)$  
 There exists a positive constant $C$ independent of $\lambda$ such that  for any $F\in L^2(K,\sigma)$,
\begin{equation}\label{esti-Poisson16}
C^{-1} \nu_\sigma(\lambda)^{-1/2}\|F\|_{L^2(K,\sigma )} \leqslant\left\|\mathcal{P}_{\sigma,\lambda}^\tau F\right\|_* \leqslant C \nu_\sigma(\lambda)^{-1/2}\|F\|_{L^2(K,\sigma)} .
\end{equation}
  Furthermore, we have
\begin{equation}\label{asym-Poisson16}
\lim _{R \rightarrow \infty} \frac{1}{R} \int_{B(R)} \diff (gK)\,\left\|\mathcal{P}_{\sigma,\lambda}^\tau F(g)\right\|_{\tau}^2 
= \frac{1}{\pi} \nu_\sigma(\lambda)^{-1} \|F\|_{L^2(K,\sigma)}^2 .
\end{equation}

$(2)$  The Poisson transform $\mathcal{P}_{\sigma,\lambda}^\tau$ is a topological isomorphism from $L^2(K,\sigma)$ onto $\mathcal{E}_{\sigma,\lambda}^2(G,\tau)$.

\end{theorem}

Moreover, we establish an $L^2$-inversion formula for the Poisson transform.

 \begin{theorem}[Inversion formula]\label{inversion}  Let $\tau=\tau_1,\cdots,\tau_{\frac{n-1}{2}}, \tau^\pm_{\frac{n}{2}}$, $\sigma\in\widehat{M}(\tau)$ and $\lambda\in\mathbb R\setminus\{0\}$.
Let $f\in \mathcal{E}^{2}_{\sigma,\lambda}(G, \tau)$. Then  its boundary value   $F $ is given by the   following inversion  formula %holds in $L^2(K/M;\sigma_q)$
\begin{equation*}
F(k) =   \pi \nu_\sigma(\lambda) \lim_{R\to\infty} \frac{1}{R}\int_{B(R)} \diff (gK)\,   e^\tau_{\sigma,\lambda}(k^{-1}g)  f(g) \quad \textit{in} \quad L^2(K,\sigma), 
\end{equation*} 
 where  
  $e^\tau_{\sigma,\lambda}(g)=\sqrt{d_{\tau,\sigma}}\mathrm{e}^{(i \lambda-\rho) H(g^{-1})} P_\sigma\tau^{-1}(\kappa(g^{-1}))$. 
 %(  {\color{red}$P_{\lambda}^\tau$ ne dépends pas de $\sigma$})
 %in \autoref{Poisson-ker}.%\kk{and $\pi_p^q \colon \Lambda^p\mathbb C^n \to \Lambda^q\mathbb C^{n-1}$ is the natural projection}.
 %where $\Phi_\lambda : V_$ is given by \autoref{phi-lam}.
\end{theorem}

%The rest of the paper  will be devoted to   the proof of  these two theorems.

The rest of the paper  will be devoted to   the proof of Theorem \ref{main-th-Poisson}, Theorem \ref{inversion}  and Theorem \ref{main-th-proj}.

\section{Asymptotic expansion of $\tau$-spherical functions} 

Let $(\tau, V_\tau)$ be    an irreducible representation of $K$.
A function $\Phi\colon G\to \operatorname{End}(V_\tau)$ with   $\Phi(e)=\operatorname{Id}$ is  called $\tau$-spherical if

\quad $(a)$ $\Phi$ is $\tau$-radial, i.e.,
$$ \Phi(k_1gk_2)=\tau(k_2)^{-1} \Phi(g) \tau(k_1)^{-1},\,\; \forall k_1, k_2\in K,\forall g\in G;$$

\quad $(b)$  $\Phi(\cdot)v$ is an eigenfunction for the algebra $\mathbb D(G,\tau)$ for one nonzero $v\in V_\tau$ (hence for all $v\in V_\tau$).\\
 %We set $\Lambda_1=\{\tau_p, 1\leq p<\frac{n-1}{2}\}$ 
 %Let $\tau\in\Lambda_1$ and $\sigma\in\widehat M(\tau)$.

Let $\tau=\tau_1,\cdots,\tau_{\frac{n-1}{2}},\tau^\pm_{\frac{n}{2}}$ and $\sigma\in\widehat{M}(\tau)$.
% Recall from \eqref{M(tau)}  that the set of $M$-types $\sigma$ that occur in $\tau_{p_{\mid M}}$ with multiplicity one are as follows 
%$$
%\widehat{M}(\tau_p)=\{\sigma_{p-1},\sigma_p\})\,\, \left(p\neq\frac{n-1}{2},\frac{n-1}{2}\right), \quad \widehat{M}(\tau_{\frac{n-1}{2}})=\{\sigma_{\frac{n-1}{2}}^+,\sigma_{\frac{n-1}{2}}^-\}, \quad \widehat{M}(\tau^\pm_\frac{n}{2})=\{\sigma^\pm_{\frac{n}{2}}\}.
%$$ 
For $\lambda  \in \mathbb{C}$,  define the Eisenstein integral
\begin{equation}\label{spherical}
 { \Phi_{\sigma,\lambda}^\tau(g)}  
 = d_{\tau,\sigma}
 \int_K \diff k\,{\rm e}^{-(i\lambda+\rho)H(g^{-1}k)}\tau(\kappa(g^{-1}k))   P_\sigma\tau(k)^{-1},
\end{equation}
where $P_\sigma$ is the projection of $V_\tau$ onto   $V_\sigma$.
Then $\Phi_{\sigma,\lambda}^\tau$ is the  $\tau$-spherical function   associated to the principal series representation $\pi_{\sigma,\lambda}=\operatorname{Ind}_{MAN}^G(\sigma\otimes e^{\rho-i\lambda}\otimes 1)$. 
Moreover, since $\tau_{\mid M}$ is multiplicity free any $\tau$-spherical function is given by \eqref{spherical}, for some $(\sigma,\lambda)\in    \widehat M(\tau)\times \mathbb{C}$.\\
For   $\Re(i\lambda)>0$, we have the following asymptotic behavior (see e.g.\cite[Proposition 5.2]{BBK}), 
\begin{equation}\label{asym-comp}
\lim_{t\rightarrow \infty}{\rm e}^{(\rho-i\lambda)t} \Phi_{\sigma,\lambda}^\tau(a_t)=d_{\tau,\sigma}\mathbf{c}(\lambda,\tau)P_\sigma, 
\end{equation}
where $\mathbf c(\lambda,\tau)$ is the generalized Harish-Chandra $c$-function given by  
\begin{equation}\label{HC-function}
\mathbf c(\lambda,\tau)
=\int_{\overline{N}}\diff\overline{n}\,{\rm e}^{-(i\lambda+\rho)H(\overline{n})}\tau(\kappa(\overline{n})) \in \mathrm{End}_M(V_\tau).
\end{equation}
Above ${\rm d}\overline n$ denotes the Haar measure on $\overline  N=\theta(N)$ with the normalization $$\int_{\overline N} {\rm e}^{-2\rho(H(\overline n))} {\rm d}\overline n=1.$$
The integral \autoref{HC-function} converges for $\lambda\in \mathbb C$ such that \(\Re(i\lambda)>0\) and it has a meromorphic continuation to \({\C}\).

 % Let $ {M}^{\prime}$ be the normalizer of $A$ in ${K}$. Then the Weyl group $W={M}^{\prime}/{M}$ has two elements. We denote by $\omega$ a representative of the non-trivial element.  
%Notice that $\omega\cdot\pi_{\sigma_q,\lambda}:=\pi_{\omega\cdot\sigma_q,-\lambda}$.  For $p$ generic and $q=p-1, p$ we have $\omega\cdot \sigma_q=\sigma_q$, hence $\Phi^{\tau_p}_{\sigma_q,-\lambda}=\Phi^{\tau_p}_{\sigma_q,\lambda}$. For $p=\frac{n-1}{2}$, $\omega\cdot\sigma_{p-1}=\sigma_{p-1}$ and $\omega\cdot\sigma_p^\pm=\sigma_p^\mp$, so $\Phi^{\tau_p}_{\sigma_{p-1},-\lambda}=\Phi^{\tau_p}_{\sigma_{p-1},\lambda}$ while $\Phi^{\tau_p}_{\sigma_p^\pm,-\lambda}=\Phi^{\tau_p}_{\sigma_p^\mp,\lambda}$. Finally, for $p=\frac{n}{2}$, $\omega\cdot \sigma_p=\sigma_p$, so $\Phi^{\tau_p^\pm}_{\sigma_p,-\lambda}=\Phi^{\tau_p^\pm}_{\sigma_p,-\lambda}$.  

By the Cartan decomposition, the $\tau$-spherical function is completely determined by its restriction to $A$. Since $\Phi_{\sigma,\lambda}(a_t)\in \operatorname{End}_M(V_\tau)$ it follows from Schur's lemma that it is scalar on each $M$-isotypic component of $V_\tau$.  Thus, there exists $\varphi_{\eta,\lambda}:\mathbb{R}\to \mathbb{C}$ such that  
$$
\Phi^\tau_{\sigma,\lambda}(a_t)=\sum_{\eta\in \widehat{M}(\tau)} \varphi_{\eta,\lambda}(t)\, \operatorname{Id}_{V_\eta}.
$$
The scalar components of $\Phi^\tau_{\sigma,\lambda}$ have been  given explicitly in \cite{Pedon}. Let us give more details.

{\bf Generic case.} %We start by the generic case. That is $1\leq p<\frac{n-1}{2}$.
In the sequel, we shall use the notation $\Phi_{q,\lambda}^p$ for $\Phi_{\sigma_q,\lambda}^{\tau_p}$ and $\varphi_{q,\lambda}$ for $\varphi_{\sigma_q,\lambda}$ where  $q= p-1,p$. %Then we have

\begin{enumerate}

 \item[$(1)$] The scalar components $\varphi_{p-1,\lambda}, \varphi_{p,\lambda}$ of $\Phi^p_{p,\lambda}$ are given by
 
 	\begin{align}
 	\varphi_{p-1,\lambda}(t)&=\phi_\lambda^{(\frac{n}{2},-\frac{1}{2})}(t), \label{phi1}\\
 	\varphi_{p,\lambda}(t)&=\frac{n}{n-p}\phi_\lambda^{(\frac{n}{2}-1,-\frac{1}{2})}(t)-\frac{p}{n-p}\cosh(t)\phi_\lambda^{(\frac{n}{2},-\frac{1}{2})}(t).\label{phi2}
 	\end{align}	
 
 \item[$(2)$]
The scalar components $\varphi_{p-1,\lambda}, \varphi_{p,\lambda}$ of $\Phi^p_{p-1,\lambda}$ are given by
\begin{align} 
 \varphi_{p-1,\lambda}(t)&=\frac{n}{p}\phi_\lambda^{(\frac{n}{2}-1,-\frac{1}{2})}(t)-\frac{n-p}{p}\cosh(t)\phi_\lambda^{(\frac{n}{2},-\frac{1}{2})}(t),\label{phi3}
 \\
 \varphi_{p,\lambda}(t)&=\phi_\lambda^{(\frac{n}{2},-\frac{1}{2})}(t).\label{phi4}
\end{align}
  \end{enumerate}
  
{\bf Special case $p=\frac{n-1}{2}$.} 
\begin{enumerate}
\item[$(1)$]
 The scalar components $\varphi_{p-1,\lambda}, \varphi^+_{p,\lambda}, \varphi^-_{p,\lambda} $ of $\Phi^{\tau_p}_{\sigma_{p-1},\lambda}$ are given by
\begin{align} 
\varphi_{p-1,\lambda}(t)&
=\frac{2n}{n-1}\phi_\lambda^{(\frac{n}{2}-1,-\frac{1}{2})}(t)-\frac{n+1}{n-1}\cosh(t)\phi_\lambda^{(\frac{n}{2},-\frac{1}{2})}(t),\label{phi1-sp1}
\\
\varphi^\pm_{p,\lambda}(t)&=\phi_\lambda^{(\frac{n}{2},-\frac{1}{2})}(t).\label{phi2-sp1}
 \end{align} 
\item[$(2)$] The scalar components $\varphi_{p-1,\lambda}, \varphi^+_{p,\lambda}, \varphi^-_{p,\lambda} $ of $\Phi^{\tau_p}_{\sigma_p^\pm,\lambda}$    are given by
%\begin{align} 
%\varphi_{p-1,\lambda}(t)&=\phi_\lambda^{(\frac{n}{2},-\frac{1}{2})}(t), \label{phi3-sp1}\\
% \varphi_{p,\lambda,}^+(t)&=\frac{2n}{n+1}\phi_\lambda^{(\frac{n}{2}-1,-\frac{1}{2})}(t)-\frac{n-1}{n+1}\cosh(t)\phi_\lambda^{(n/2-1,-1/2)}(t)\pm\frac{i2\lambda}{n+1}\sinh(t) \frac{2n}{n+1}\phi_\lambda^{(\frac{n}{2}-1,-\frac{1}{2})}(t),  bla \label{phi4-sp1}\\
%\varphi_{p,\lambda}^-&=\frac{2n}{n+1}\phi_\lambda^{(\frac{n}{2}-1,-\frac{1}{2})}(t)-\frac{n-1}{n+1}\cosh(t)\phi_\lambda^{(\frac{n}{2}-1,-\frac{1}{2})}(t)\mp\frac{i2\lambda}{n+1}\sinh(t) \frac{2n}{n+1}\phi_\lambda^{(\frac{n}{2}-1,-\frac{1}{2})}(t). \label{phi5-sp1}
% \end{align}
 \begin{align} 
 \varphi_{p-1,\lambda}(t)&=\phi_\lambda^{(\frac{n}{2},-\frac{1}{2})}(t), \label{phi3-sp1}
 \\
 \varphi_{p,\lambda,}^+(t)&=
 \frac{2n}{n+1}\phi_\lambda^{(\frac{n}{2}-1,-\frac{1}{2})}(t)
 -\frac{n-1}{n+1}\cosh(t)\phi_\lambda^{(n/2,-1/2)}(t)
 \pm\frac{i2\lambda}{n+1}\sinh(t) \phi_\lambda^{(\frac{n}{2},-\frac{1}{2})}(t),  \label{phi4-sp1}
 \\
\varphi_{p,\lambda}^-(t)&=
\frac{2n}{n+1}\phi_\lambda^{(\frac{n}{2}-1,-\frac{1}{2})}(t)
-\frac{n-1}{n+1}\cosh(t)\phi_\lambda^{(\frac{n}{2},-\frac{1}{2})}(t)\mp\frac{i2\lambda}{n+1}\sinh(t)  \phi_\lambda^{(\frac{n}{2},-\frac{1}{2})}(t). \label{phi5-sp1}
 \end{align}

 \end{enumerate}

 {\bf Special case $p=\frac{n}{2}$.}  The scalar component $\varphi_{p,\lambda}^\pm$ of $\Phi^{\tau_p^\pm}_{\sigma_p}$ is given by
 \begin{equation}\label{phi-sp2}
 	\varphi_{p,\lambda}^+(t)=\varphi_{p,\lambda}^-(t)=\cosh\left(\frac{t}{2}\right) \phi_{2\lambda}^{(\frac{n}{2}-1,\frac{n}{2}+1)}\left(\frac{t}{2}\right).
 \end{equation}

Above $\phi_\lambda^{(\alpha,\beta)}$ denotes the Jacobi function
\begin{equation*}%\label{jacobi}
 \phi_\lambda^{(\alpha,\beta)}(t)={}_2F_1\left(\frac{i\lambda+\alpha+\beta+1}{2},\frac{-i\lambda+\alpha+\beta+1}{2};\alpha+1; -\sinh^2 t\right),  
\end{equation*}
where  $\alpha, \beta, \lambda\in \mathbb{C}$ with $\alpha\neq -1,-2, \ldots$,  (see, e.g. \cite{Pedon}\cite{Ko}).\\
We recall below some known results on Jacobi functions that will be needed in the sequel (see \cite{Ko, FJ}).\\

For $\lambda\notin i\mathbb{Z}$,   $\phi_\lambda^{\alpha,\beta}$ has the following development
  \begin{equation}\label{key-F-2}
\phi^{(\alpha, \beta)}_\lambda(t)=c_{\alpha, \beta}( \lambda) \Psi_{\lambda}^{\alpha, \beta}(t)+ c_{\alpha, \beta}(- \lambda) \Psi_{- \lambda}^{\alpha, \beta}(t), 
\end{equation}
where %$c_{\alpha, \beta}(\lambda)$ is given by 
$$
c_{\alpha, \beta}(\lambda)=\frac{2^{-i \lambda+\alpha+\beta+1} \Gamma(\alpha+1) \Gamma(i \lambda)}{\Gamma\left(\frac{i \lambda+\alpha+\beta+1}{2}\right) \Gamma\left(\frac{i \lambda+\alpha-\beta+1}{2}\right)},
$$
and  %$\Psi_\lambda^{\alpha,\beta}(t)$  is given
\begin{align}\begin{split}\label{k}
\Psi_\lambda^{\alpha, \beta}(t)&=(2 \sinh t)^{i \lambda-\alpha-\beta-1}{ }_2 F_1\left(\frac{\alpha+\beta+1-i \lambda}{2}, \frac{-\alpha+\beta+1-i \lambda}{2} ; 1-i \lambda ;-\frac{1}{\sinh ^2 t}\right)\\
&={\rm e}^{(i\lambda-\alpha-\beta-1)t}(1+o( 1)) \quad \textit{as}\quad t\to \infty.
\end{split}\end{align}
Furthermore,
there exists a constant $C>0$ such that
\begin{equation*}%\label{theta-1}
\Psi_\lambda^{\alpha, \beta}(t)=e^{\left(i \lambda-\alpha-\beta-1 \right) t}\left(1+\mathrm{e}^{-2 t} \Theta(\lambda,t)\right), \text { with } \; \left|\Theta(\lambda,t)\right| \leqslant C,
\end{equation*}
  for all $\lambda \in \mathbb{R}$ and all $t \geqslant 1$.

 % For $q\in\{p-1,p\}$  we set $\rho_q=\rho-q$.
\begin{lemma}
Let $\tau=\tau_1,\cdots, \tau_{\frac{n-1}{2}}, \tau^\pm_{\frac{n}{2}}$ and $\sigma\in \widehat{M}(\tau)$. For $\lambda\in \mathbb{C}$ such that $\Re(i\lambda)>0$, there exists a meromorphic $\mathbb{C}$-valued function $c_\sigma(\lambda,\tau)$ such that  
\begin{equation}\label{comp-scal-c}
d_{\tau,\sigma} \mathbf{c}(\lambda,\tau)P_\sigma =c_\sigma(\lambda,\tau)P_{\sigma}.
\end{equation}
Moreover,
 
$(1)$ in the generic case % generic   %$$1\leq p<\frac{n-1}{2}$.\\ 
 \begin{align}\label{c27-03}
c_\sigma(\lambda,\tau)= 
\begin{cases}
	\displaystyle\frac{i\lambda+\rho-p}{2(n-p)}c_{\frac{n}{2},-\frac{1}{2}}(\lambda) & \text{for $(\tau,\sigma)=(\tau_p,\sigma_p)$}   \\
\displaystyle\frac{i\lambda-\rho+p-1}{2p}c_{\frac{n}{2},-\frac{1}{2}}(\lambda) & \text{for $(\tau,\sigma)=(\tau_p,\sigma_{p-1})$} 
\end{cases}
\end{align}

$(2)$ in the special cases $p=\frac{n-1}{2}, \frac{n}{2}$, %The special cases. $\tau\in \{\tau_\frac{n-1}{2}, \tau^\pm_\frac{n}{2}\}$.
\begin{align*}
	c_\sigma(\lambda,\tau)= 
	\begin{cases}
	 \displaystyle\frac{i\lambda-1}{n-1}c_{\frac{n}{2},-\frac{1}{2}}(\lambda) & \text{for   $(\tau,\sigma)=(\tau_\frac{n-1}{2},\sigma_{\frac{n-1}{2}-1})$}\\
		\displaystyle\frac{2i\lambda}{n+1}c_{\frac{n}{2},-\frac{1}{2}}(\lambda) & \text{for   $(\tau,\sigma)=(\tau_\frac{n-1}{2},\sigma^\pm_{\frac{n-1}{2}})$}\\
		\displaystyle\frac{1}{2}c_{\frac{n}{2}-1,\frac{n}{2}+1}(2\lambda) & \text{for   $(\tau,\sigma)=(\tau^\pm_\frac{n}{2},\sigma_{\frac{n}{2}})$}\\
	\end{cases}
\end{align*}

%\begin{align}
%c_\sigma(\lambda,\tau)=\left\lbrace
%\begin{array}{ll}
%\frac{2i\lambda}{n+1}c(\lambda) & \mbox{if}\quad  (\tau,\sigma)=(\tau_\frac{n-1}{2},\sigma^\pm_{\frac{n-1}{2}})\\
%\frac{1}{2}c_{\frac{n}{2}-1,\frac{n}{2}+1}(\lambda) & \mbox{if} \quad (\tau,\sigma)=(\tau^\pm_\frac{n}{2},\sigma_{\frac{n}{2}})
%\end{array}
%\right.
%\end{align}
\end{lemma}
\begin{proof}
$(1)$ Suppose $p$ generic and $q=p-1, p$  and let $\lambda\in \mathbb{C}$ with  $\Re(i\lambda)>0$. A direct computation using \eqref{key-F-2} and \eqref{k} together with  the identity $c_{\frac{n}{2}-1,-\frac{1}{2}}(\lambda)=\frac{i\lambda+\rho}{2n} c_{\frac{n}{2},-\frac{1}{2}}(\lambda) $ % a direct calculation 
gives   
$$
\lim_{t\to\infty}{\rm e}^{(\rho-i\lambda)t}\Phi_{q,\lambda}^p(a_t)=c_q(\lambda,p)\operatorname{Id}_{\Lambda^{q}\mathbb C^{n-1}}, 
$$
where
 \begin{equation*}
c_{p-1}(\lambda,p)    
 =\frac{i\lambda- \rho+p-1  }{2p}c_{\frac{n}{2},-\frac{1}{2}}(\lambda),  \quad  c_p(\lambda,p)   
 =\frac{i\lambda+ \rho-p }{2(n-p)} c_{\frac{n}{2},-\frac{1}{2}}(\lambda).
\end{equation*}
   %$c(\lambda)=c_{\frac{n}{2}-1,-\frac{1}{2}}(\lambda)$
%$$c(\lambda)= 2^{\rho-i\lambda} \frac{\Gamma(i\lambda) \Gamma\left(\rho+\frac{1}{2}\right)}
%{\Gamma\left(\frac{i\lambda +\rho}{2} \right) \Gamma\left(\frac{i\lambda +\rho+1}{2} \right)}.$$
Then the identity \eqref{comp-scal-c}
% $$
%d_{\tau_p,\sigma_q} \textbf{c}(\lambda,\tau_p)P_{\sigma_q}=c_q(\lambda,p)P_{\sigma_q},
%$$
  follows  from  \eqref{asym-comp}. % \kk{on utilise  \eqref{asym-comp} qui n'est pas valide pour $\lambda$ réel ?? il faut faire une extension méromorphe ....}

$(2)$   The proof for the special cases runs as before.
%In the same manner we get  
%$$
% d_{\tau_{\frac{n-1}{2}},\sigma^\pm_{\frac{n-1}{2}}} 
% \textbf{c}(\lambda,\tau_{\frac{n-1}{2}})P_{\sigma^\pm_{\frac{n-1}{2}}}=\frac{i2\lambda}{n+1}c_{\frac{n}{2},-\frac{1}{2}}(\lambda)P_{\sigma^\pm_{\frac{n-1}{2}}},
%$$
%$$
 %d_{\tau_{\frac{n-1}{2}},\sigma_{\frac{n-1}{2}-1}} 
%\textbf{c}(\lambda,\tau_{\frac{n-1}{2}})P_{\sigma_{\frac{n-1}{2}-1}}=\frac{i\lambda-1}{n-1}c_{\frac{n}{2},-\frac{1}{2}}(\lambda) P_{\sigma_{\frac{n-1}{2}-1}}
%$$
%and 
%$$
% \textbf{c}(\lambda,\tau^\pm_{\frac{n}{2}})=\frac{1}{2}c_{\frac{n}{2}-1,\frac{n}{2}+1}(2\lambda)\operatorname{Id}.
%$$
\end{proof}

\begin{remark}
Let $\lambda\in \mathbb{R}\setminus\{0\}$, $\tau\in \Lambda$ and  $\sigma\in \widehat{M}(\tau)$. Then as expected the Plancherel density is related to the $c$-function $c_\sigma(\lambda,\tau)$ by   the following formula
\begin{align}\label{density}
\nu_{\sigma}(\lambda)=\frac{1}{2\pi}d_{\tau,\sigma}\mid c_{\sigma}(\lambda,\tau)\mid^{-2}.
\end{align}
\end{remark}

  % \kk{  \sout{Notice that the  asymptotic \eqref{asym-comp} is no longer  true for $\lambda$ real   because of the oscillating terms at infinity in $  \Phi_{\sigma,\lambda}^\tau(a_t)$.}}

 The main result of this section is to establish, for nonzero real parameter $\lambda$,   an asymptotic expansion of the translated $\tau$-spherical functions $\Phi_{\sigma,\lambda}^\tau(g^{-1}x)$, ($x, g\in G$).  
  Before stating  the result let us say     that $f_1, f_2 \in B^*(G,\tau)$ are asymptotically equivalent, and we denote $f_1\simeq f_2$, if
\begin{equation}\label{asym-means} 
\lim_{R\to \infty} \frac{1}{R}\int_{B(R)} \diff(gK)\,\|f_1(g)-f_2(g)\|_\tau^2   =0.
 \end{equation}

  \begin{proposition}
    Let  $\tau=\tau_1,\cdots,\tau_{\frac{n-1}{2}}, \tau^\pm_{\frac{n}{2}}$, $\sigma\in \widehat{M}(\tau)$ and $\lambda\in\R\setminus\{0\}$. Then for any $g\in G$ we have the following asymptotic expansion   in $B^*(G,\tau)$,
 \begin{equation}\label{main-formula1}
 \Phi_{\sigma,\lambda}^\tau (g)v \simeq  \tau^{-1}(k_2(g)) \sum_{s\in W} e^{(is\lambda-\rho)A^+(g)}  d_{\tau,s\sigma}\mathbf{c}(s\lambda,\tau)   P_{s\sigma}\tau^{-1}(k_1(g))  v, 
 \end{equation}
 where $v\in V_\tau$ and $g=k_1(g)e^{A^+(g)}k_2(g)$.
  \end{proposition}
  
  \begin{proof}  
 Let us first  show that the right hand-side of \eqref{main-formula1} belongs to $ B^\ast(G,\tau)$. %For short put $ c_\sigma(\lambda,\tau)=\frac{\dim \tau}{\dim\sigma}\mathbf{c}(\lambda,\tau)$ 
 Since $  d_{\tau,\sigma}\mathbf{c}(\lambda,\tau)P_\sigma=c_\sigma(\lambda,\tau)P_\sigma$ and $\dim \sigma=\dim s\sigma$ for any $s\in W$, 
 we have
\begin{align*}\begin{split}
\frac{1}{R}
\int_{B(R)}\diff (gK)\parallel [\tau^{-1}(k_2(g)) 
&\sum_{s\in W} e^{(is\lambda-\rho)A^+(g)} d_{\tau,s\sigma}\,\mathbf{c}(s\lambda,\tau)   P_{s\sigma}\tau^{-1}(k_1(g))]v\parallel^2\\
& \leqslant \sum_{s\in W}\frac{1}{R}\int_{B(R)}\diff (gK)\,{\rm e}^{-2\rho A^+(g)}\mid c_{{s}\sigma}(s\lambda,\tau)\mid^2 \parallel P_{{s}\sigma} \tau^{-1}(k_1(g))]v\parallel^2\\
&\leqslant
 2 \mid c_{\sigma}(\lambda,\tau)\mid^2 \frac{d_\sigma}{d_\tau} \parallel v\parallel^2
\frac{1}{R} \int_0^R \diff t\,{\rm e}^{-2\rho t} (2\sinh t)^{2\rho}\\
&\leqslant 2 \mid c_{\sigma}(\lambda,\tau)\mid^2 \frac{d_\sigma}{d_\tau} \parallel v\parallel^2,
\end{split}\end{align*}
for any $R>0$, what was required to be shown. Above we have also used the fact $|c_{s\sigma}(s\lambda,\tau)|^2=|c_{\sigma}(\lambda,\tau)|^2$.

 Now we prove the asymptotic expansion \eqref{main-formula1}. We have 
\begin{align*}
\begin{split}
\int_{B(R)}\diff (gK) &\left\| \Phi_{\sigma,\lambda}^\tau(g)v-\tau^{-1}(k_2(g)) \sum_{s\in W} e^{(is\lambda-\rho)A^+(g)} d_{\tau,s\sigma}\mathbf{c}(s\lambda, \tau)   P_{s\sigma}\tau^{-1}(k_1(g))v\right\|^2\\
&=\int_0^R   \int_K  \left\|\left[ \Phi_{\sigma,\lambda}^\tau(a_t)- \sum_{s\in W} e^{(is\lambda-\rho)t}d_{\tau,s\sigma}\mathbf{c}(s\lambda,\tau)   P_{s\sigma}\right]\tau^{-1}(k)v\right\|^2   (2\sinh t)^{2\rho}\diff k\diff t .
\end{split}
\end{align*}
We claim that for any $\lambda\in\R\setminus\{0\}$ there exists  a positive constant $C_\lambda$ such that  for every $t>0$ we have
\begin{equation}\label{asympt}
\left\|\left[\Phi^\tau_{\sigma,\lambda}(a_t)-\sum_{s\in W} e^{(is\lambda-\rho)t} d_{\tau,s\sigma} \mathbf{c}(s\lambda, \tau) P_{s\sigma}\right]\tau^{-1}(k)v\right\|\leqslant  C_\lambda{\rm e}^{-\rho t -t} \|v\|.
\end{equation}
Thus, from the definition \eqref{asym-means}, it is clear that \eqref{asympt} implies \eqref{main-formula1}.
%that  
%\begin{align*}
%\lim_{R\to \infty}\frac{1}{R}\int_{B(R)}\diff x 
%&\left\| \Phi_{\sigma,\lambda}^\tau(x)v-d_{\tau,\sigma}\tau^{-1}(k_2(x)) \sum_{s\in W} e^{(is\lambda-\rho)A^+(x)} \mathbf{c}(s\lambda, \tau)   P_{s\sigma}\tau^{-1}(k_1(x))v\right\|^2=0.
%\end{align*}
So, we are left to prove the estimate \eqref{asympt}.
By   continuity it is sufficient to prove it  for $t\geqslant 1$.

{\bf Case : $p$ generic and $q=p$.} %First,     we have $s\sigma_p=\sigma_p$, for any $s\in W$.
 Let us  introduce the temporary  notation $\Phi_{p,\lambda}^p$ for $\Phi_{\sigma_p,\lambda}^{\tau_p}$.  According to \eqref{key-F-2} we may rewrite the scalar components of $\Phi^p_{p,\lambda}$ (see \eqref{phi1}, \eqref{phi2}) as 
\begin{align*}
	\varphi_{p-1,\lambda}(t)&=\sum_{s\in W}c_{\frac{n}{2},-\frac{1}{2}}(s\lambda){\rm e}^{(is\lambda-\rho-1)t}(1+{\rm e}^{-2t}\Theta_1(s\lambda,t))\\
	\varphi_{p,\lambda}(t)&=\frac{n}{n-p}\sum_{s\in W}c_{\frac{n}{2}-1,-\frac{1}{2}}(s\lambda){\rm e}^{(is\lambda-\rho)t}(1+{\rm e}^{-2t}\Theta_2(s\lambda,t))\\
&-\frac{p}{n-p} (1+{\rm e}^{-2t})\sum_{s\in W}\frac{1}{2}c_{\frac{n}{2},-\frac{1}{2}}(s\lambda){\rm e}^{(is\lambda-\rho)t}(1+{\rm e}^{-2t}\Theta_3(s\lambda,t) )
\end{align*}
with $\mid \Theta_i(\lambda,t)\mid\leqslant C_i$ ($i=1, 2,3$) for all $t\geqslant 1$. Now, using the identity 
 $$
 c_{\frac{n}{2}-1,-\frac{1}{2}}(\lambda)=\frac{i\lambda+\rho}{2n}c_{\frac{n}{2},-\frac{1}{2}}(\lambda),
$$   
 a direct calculation yields  to
\begin{align*}
\varphi_{p,\lambda}(t)=\frac{n}{(n-p)}\sum_{s\in W}\frac{is\lambda+\rho-p}{is\lambda+\rho} c_{\frac{n}{2}-1,-\frac{1}{2}}(s\lambda){\rm e}^{(is\lambda-\rho)t}+{\rm e}^{-(\rho+2)t}\Theta(\lambda,p,t),
\end{align*}
with $\mid \Theta((\lambda,p,t)\mid \leqslant C$ for all $t\geqslant 1$ and $\lambda\in \mathbb{R}$.\\
Next, let $v\in V_{\tau_p}$ and write $v=v_{p-1}+v_p$, then $\Phi^p_{p,\lambda}(a_t)v=\varphi_{p-1,\lambda}(t)v_{p-1}+ \varphi_{p,\lambda}(t)v_{p}$. Since  $c_p(\lambda,p)=\frac{i\lambda+\rho-p}{i\lambda+\rho} c_{\frac{n}{2}-1,-\frac{1}{2}}(\lambda)$ (see \eqref{c27-03}), we get 
 $$\begin{aligned}
\Phi^p_{p,\lambda}(a_t)v%&=\varphi_{p-1,\lambda}(t)v_{p-1}+ \varphi_{p,\lambda}(t)v_{p}\\
&={\rm e}^{-t}\sum_{s\in W}c_{\frac{n}{2},-\frac{1}{2}}(s\lambda){\rm e}^{(is\lambda-\rho)t}(1+{\rm e}^{-2t}\Theta_1(s\lambda,t))v_{p-1}\\
&+\frac{n}{(n-p)}[\sum_{s\in W}c_p(\lambda,p){\rm e}^{(is\lambda-\rho)t}+{\rm e}^{-(\rho+2)t}\Theta(\lambda,p,t)]v_p
\end{aligned}
$$
from which we deduce that 
\begin{align*}
\Phi^p_{p,\lambda}(a_t)v-\frac{n}{(n-p)}[\sum_{s\in W}c_p(s\lambda,p){\rm e}^{(is\lambda-\rho)t} ]v_p ={\rm e}^{-t}{\rm e}^{-\rho t}K(\lambda,p,t)v,
\end{align*}
where  $K(\lambda,p,t)\in \mathrm{End}(V_{\tau_p})$ with $\parallel K(\lambda,p,t)\parallel\leqslant C_\lambda$ for all $t\geqslant 1$. 
 As $s\sigma_p=\sigma_p$, for any $s\in W$, $\frac{n}{n-p}=d_{\tau_p,\sigma_p}$ and $d_{\tau_p,\sigma_p}\textbf{c}(\lambda,\tau_p)={c}_p(\lambda,\tau_p)P_{\sigma_p}$, 
we obtain 
 
$$
\Phi^p_{p,\lambda}(a_t)v-\sum_{s\in W} {\rm e}^{(is\lambda-\rho)t} d_{\tau_p,\sigma_p} \textbf{c}(s\lambda,\tau_p)P_{\sigma_p}v={\rm e}^{-t}{\rm e}^{-\rho t}K(\lambda,p,t)v,
$$
and consequently
$$
\parallel \Phi^p_{p,\lambda}(a_t)-\sum_{s\in W} {\rm e}^{(is\lambda-\rho)t}\ d_{\tau_p,\sigma_p} \textbf{c}(s\lambda,\tau_p)P_{\sigma_p}\parallel\leqslant C_\lambda {\rm e}^{-t}{\rm e}^{-\rho t}
$$
for all $t\geqslant 1$, as to be shown.

{\bf  Case $p$ generic and $q=p-1$.} The proof is similar to the previous case, so we omit it. 

\textbf{Special case $p=\frac{n-1}{2}$.}
For $\sigma=\sigma_{\frac{n-1}{2}-1}$ the proof is similar to the generic case. We shall now deal with the case $\sigma=\sigma_{\frac{n-1}{2}}^+$.  Using \eqref{phi4-sp1}-\eqref{phi5-sp1}   and following the same method as before, we get
$$
\varphi_{p,\lambda}^\pm(t)=\frac{1}{n+1}\sum_{s\in W}(is\lambda\pm i\lambda)c_{\frac{n}{2},-\frac{1}{2}}(s\lambda){\rm e}^{(is\lambda-\rho)t}+{\rm e}^{-\rho t}{\rm e}^{-t}I^\pm(\lambda,t),
$$
with $\mid I^\pm(\lambda,t)\mid\leqslant C_\lambda$, for all $t\geqslant 1$. Therefore 
$$
\Phi^{\tau_{\frac{n-1}{2}}}_{\sigma_{\frac{n-1}{2}}^+}(a_t)v=\frac{2i\lambda}{n+1}c_{\frac{n}{2},-\frac{1}{2}}(\lambda){\rm e}^{(i\lambda-\rho)t}v_p^+-\frac{2i\lambda}{n+1}c_{\frac{n}{2},-\frac{1}{2}}(-\lambda){\rm e}^{(-i\lambda-\rho)t}v_p^-+ {\rm e}^{-t}{\rm e}^{-\rho t}I^+(\lambda,t),
$$
and the estimate \eqref{asympt} follows.\\
The  case $\sigma=\sigma_{(n-1)/2}^-$ is is obtained in the same way since    $\Phi_{\sigma_{{n-1}/{2}}^-,\lambda}^{\tau_p}$ and $\Phi_{\sigma_{{n-1}/{2}}^+,\lambda}^{\tau_p}$ have the same scalar components.

%the result follows from the following remark.
%If $\psi_{p-1,\lambda},\psi_{p,\lambda}^+,\psi_{p,\lambda}^-$ are the components of $\Phi_{\sigma_{(n-1)/2}^-,\lambda}$ then 
%$$
%\psi_{p-1,\lambda}=\varphi_{p-1,\lambda}  \quad \psi_{p,\lambda}^\mp=\varphi_{p,\lambda}^\pm. 
%$$

\textbf{Special case $p=\frac{n}{2}$.} It is similar to the generic case using \eqref{phi-sp2}, so we omit details.
\end{proof}

Next, we will prove the following asymptotic for the translated $\tau$-spherical functions.

\begin{proposition}[Key formula]\label{key-formula} 
Let $\tau=\tau_1,\cdots,\tau_{\frac{n-1}{2}}, \tau^\pm_{\frac{n}{2}}$, and $\sigma\in\widehat M(\tau)$. 
  Let $\lambda\in\R\setminus\{0\}$  and let $g\in G$ be fixed. Then for any $x\in G$ we have the following expansion  in $B^*(G,\tau)$,
 $$ \Phi_{\sigma,\lambda}^\tau (g^{-1}x)v \simeq  \tau(k_2(x))^{-1}\sum_{s\in W} e^{(is\lambda-\rho)A^+(x)} 
 e^{(is\lambda-\rho)H(g^{-1}k_1(x))}d_{\tau,s\sigma}\mathbf{c}(s\lambda,\tau) P_{s\sigma} \tau(\kappa(g^{-1}k_1(x)))^{-1}v,
 $$
 where $v\in V_\tau$ and $x=k_1(x) e^{A^+(x)}k_2(x)$.
 \end{proposition}
 
 \begin{proof}
  It follows from    
the estimate \eqref{main-formula1}, that 
\begin{align}\label{Phi11}
\Phi_{\sigma,\lambda}^\tau(g^{-1}x)v\simeq \tau^{-1}(k_2(g^{-1}x))\sum_{s\in W} {\rm e}^{(is\lambda-\rho)A^+(g^{-1}x)} d_{\tau,s\sigma}\mathbf{c}(\tau,s\lambda) P_{s\sigma}\tau^{-1}(k_1(g^{-1}x))v.
\end{align}
We rewrite the right-hand side of \eqref{Phi11} in the form
\begin{align*}
  \tau^{-1}(k_2(x))\sum_{s\in W}       {\rm e}^{(is\lambda-\rho){A^+(x)+H(g^{-1}k_1(x))}} d_{\tau,s\sigma}\mathbf{c}(s\lambda,\tau)P_{s\sigma}\tau^{-1}(\kappa(g^{-1}k_1(x))v +R_g(x),
\end{align*}
where 
$$\begin{aligned} %\label{Phi21}
R_g(x) &= \tau^{-1}(k_2(g^{-1}x))\sum_{s\in W} 
  {\rm e}^{(is\lambda-\rho)A^+(g^{-1}x)} 
d_{\tau,s\sigma}\mathbf{c}(s\lambda,\tau)P_{s\sigma} \tau^{-1}(k_1(g^{-1}x))v\\
&-\tau^{-1}(k_2(x))\sum_{s\in W}     {\rm e}^{(is\lambda-\rho)(A^+(x)+H(g^{-1}k_1(x)))} 
d_{\tau,s\sigma} \mathbf{c}(s\lambda,\tau)P_{s\sigma}\tau^{-1}(\kappa(g^{-1}k_1(x))v. 
\end{aligned}$$
So, we have to show that $R_g\simeq 0$. To do so, we need the following lemmas, which will be proved in the appendix (see Lemma \ref{cliffA} and Lemma \ref{key-fo11A}). 
 
\begin{lemma}\label{cliff}   For any $g, x\in G$ we have
\begin{align}\label{AH}
A^+(gx)=A^+(x)+H(gk_1(x))+E(g,x),
\end{align}
where the function $E$ satisfies the estimate
\begin{align}\label{AHapp}
0<E(g,x)\leqslant {\rm e}^{2(A^+(g)-A^+(x)}.
\end{align}
 \end{lemma}

\begin{lemma}\label{key-fo11}
  For any $g\in G$  we have
\begin{equation*}
\lim_{R\to\infty}   \tau^{-1}(k_2(g e^{RH_0})P_\sigma\tau^{-1}(k_1(g e^{RH_0})=P_\sigma\tau^{-1}(\kappa(g)).
 \end{equation*}
  \end{lemma}

Now, use the identity \eqref{AH}  
 to rewrite $R_g(x)$ as $R_g(x)=I_g(x)+J_g(x)$, where
 $$\begin{aligned}
 I_g(x)&=
 \tau^{-1}(k_2(g^{-1}x)) \sum_{s\in W}                          {\rm e}^{(is\lambda-\rho)(A^+(x)+H(g^{-1}k_1(x))}\times\\
 &\hspace{4cm} \times \Bigl[{\rm e}^{(is\lambda-\rho)E(g^{-1},x)}-1\Bigr] 
 d_{\tau,s\sigma}\mathbf{c}(s\lambda,\tau)P_{s\sigma} \tau^{-1}( k_1(g^{-1}x)) v,\\
% \end{aligned}
% $$
 %and
 %$$ \begin{aligned}
 	J_g(x)&=
 \sum_{s\in W}  {\rm e}^{(is\lambda-\rho)(A^+(x)+H(g^{-1}k_1(x))}\times\\
& \times\Bigl[\tau^{-1}(k_2(g^{-1}x))d_{\tau,s\sigma}\mathbf{c}(s\lambda,\tau)P_{s\sigma} \tau^{-1}(k_1(g^{-1}x)) 
-\tau^{-1}(k_2(x))d_{\tau,s\sigma}\mathbf{c}(s\lambda,\tau)P_{s\sigma}\tau^{-1}(\kappa(g^{-1}k_1(x)))\Bigr]v.
 \end{aligned}
 $$
We shall  prove that $I_g\simeq 0$ and $J_g\simeq0$.\\
Since 
\begin{equation}\label{20-aout}
	d_{\tau,s\sigma}\mathbf{c}(s\lambda,\tau)P_{s\sigma}=c_{s\sigma}(s\lambda,\tau)P_{s\sigma} \; \text{ and } \; \mid c_{s\sigma}(s\lambda,\tau)\mid=\mid c_{\sigma}(\lambda,\tau)\mid
\end{equation}
%$ d_{\tau,s\sigma}\mathbf{c}(s\lambda,\tau)P_{s\sigma}=c_{s\sigma}(s\lambda,\tau)P_{s\sigma}$ and $\mid c_{s\sigma}(s\lambda,\tau)\mid=\mid c_{\sigma}(\lambda,\tau)\mid$, we have
we have
$$\begin{aligned}
&\frac{1}{R}\int_{B(R)}\diff(xK)\,\parallel I_g(x)\parallel^2  
\leqslant \\
&2  \parallel v\parallel^2  \mid c_\sigma(\lambda,\tau)\mid^2
\sum_{s\in W}  
\frac{1}{R}\int_{B(R)}\diff(xK)\,{\rm e}^{-2\rho(A^+(x)+H(g^{-1}k_1(x))}\mid {\rm e}^{(is\lambda-\rho)E(g,x)}-1\mid^2 .
\end{aligned}$$
 From Lemma \eqref{cliff}    we have the estimate $|E(g,x)|\leqslant c_g {\rm e}^{-2A^+(x)}$ and therefore 
 $$ |{\rm e}^{(i\lambda-\rho)E(g,x)}-1|^2\leqslant (\lambda^2+\rho^2) c_g {\rm e}^{-2A^+(x)}.$$ 
Put $\gamma=2 (\lambda^2+\rho^2)c_g |c_\sigma(\lambda,\tau)|^2\parallel v  \parallel^2$. Then we have
 $$\begin{aligned}
&\frac{1}{R}\int_{B(R)}\diff(xK)\,\parallel I_g(x)\parallel^2  
\leqslant \gamma\frac{1}{R}\int_{B(R)} \diff(xK)\,{\rm e}^{-2\rho(A^+(x)+H(g^{-1}k_1(x))} {\rm e}^{-2A^+(x)}\\
&\leqslant \gamma\frac{1}{R}\int_0^R {\rm e}^{-(2\rho+2)t}(2\sinh t)^{2\rho}{\rm d}t\left(\int_K {\rm e}^{-2\rho(Hg^{-1}k)}{\rm d}k\right)
%&2  (\lambda^2+\rho^2)c_g\parallel v  \parallel^2  |c_\sigma(\lambda,\tau)|^2 \frac{1}{R}\int_{[0,R]\times K}{\rm d}t\,{\rm d}k\,\, \hb{{\rm e}^{-2 t}}{\rm e}^{-2\rho t}(2\sinh t)^{2\rho} {\rm e}^{-2\rho H(g^{-1}k)},
\end{aligned}$$
from which we get $I_g\simeq 0$.

 Next we shall prove that $J_g\simeq 0$.    Using \eqref{20-aout} and the Cartan decomposition $x=k a_t h$ we easily see that 
\begin{align*}
\begin{split}
& \frac{1}{R}\int_{B(R)}\diff(xK) \parallel J_g(x)\parallel^2 \\
&  \leqslant 2|c_{\sigma}(\lambda,\tau)|^2| \sum_{s\in W}    \frac{1}{R}
 \int_0^R \int_K \diff k\diff t \, {\rm e}^{-2\rho t}(2\sinh t)^{2\rho} {\rm e}^{-2\rho(H(g^{-1}k))}\times\\
&\hspace{3cm}\times \parallel \left( \tau^{-1}(k_2(g^{-1}ka_t)P_{s\sigma}\tau^{-1}(k_1(g^{-1}ka_t))-P_{s\sigma}\tau^{-1}(\kappa(g^{-1}k) \right)v\parallel^2   \\
&\leqslant
 2|c_{\sigma}(\lambda,\tau)|^2| \sum_{s\in W}   \int_0^1\int_K \diff k\diff t\, {\rm e}^{-2\rho(H(g^{-1}k))}\times\\
& \hspace{3cm}\times \parallel \left( \tau^{-1}(k_2(g^{-1}ka_{tR})P_{s\sigma}\tau^{-1}(k_1(g^{-1}ka_{tR}))-P_{s\sigma}\tau^{-1}(\kappa(g^{-1}k) \right)v\parallel^2.  \\
\end{split}
\end{align*}
By Lemma \ref{key-fo11} we have 
$$
\lim_{R\to \infty}
%\parallel \tau^{-1}(k_1(g^{-1}ka_{Rt})k_2(g^{-1}ka_{Rt}))-\tau^{-1}(\kappa(g^{-1}k))\parallel^2
\parallel \left( \tau^{-1}(k_2(g^{-1}ka_{tR})P_{s\sigma}\tau^{-1}(k_1(g^{-1}ka_{tR}))-P_{s\sigma}\tau^{-1}(\kappa(g^{-1}k) \right)v\parallel^2
=0.
$$
Since 
$$\begin{aligned}
&{\rm e}^{-2\rho H(g^{-1}k)}
\parallel \left( \tau^{-1}(k_2(g^{-1}ka_{tR})P_{s\sigma}\tau^{-1}(k_1(g^{-1}ka_{tR}))-P_{s\sigma}\tau^{-1}(\kappa(g^{-1}k) \right)v\parallel^2\\
&\leqslant 2 {\rm e}^{-2\rho(H(g^{-1}k))}\|v\|^2	
\end{aligned}$$
%$${\rm e}^{-2\rho H(g^{-1}k)}
%\parallel \left( \tau^{-1}(k_2(g^{-1}ka_{tR})P_{s\sigma}\tau^{-1}(k_1(g^{-1}ka_{tR}))-P_{s\sigma}\tau^{-1}(\kappa(g^{-1}k) \right)v\parallel^2
%\leq 2 {\rm e}^{-2\rho(H(g^{-1}k))}\|v\|^2$$
 and noting that $\int_K {\rm e}^{-2\rho H(g^{-1}k)} {\rm d}k=1$, we   obtain by  Lebesgue's dominated convergence theorem that
\begin{align*}
\lim_{R\rightarrow \infty}\int_K  \diff k\, {\rm e}^{-2\rho H(g^{-1}k)}
%\parallel \tau^{-1}(k_1(g^{-1}ka_{Rt})k_2(g^{-1}ka_{Rt}))-\tau^{-1}(\kappa(g^{-1}k))\parallel^2 
\parallel \left( \tau^{-1}(k_2(g^{-1}ka_{tR})P_{s\sigma}\tau^{-1}(k_1(g^{-1}ka_{tR}))-P_{s\sigma}\tau^{-1}(\kappa(g^{-1}k) \right)v\parallel^2
 =0.
\end{align*}
Thus $\lim_{R\to\infty}\frac{1}{R}\int_{B(R)}\diff (xK)\parallel J_g(x)\parallel^2 =0$, and   the proposition follows.
\end{proof}

 For any $\lambda\in\R$, $g\in G$ and $v\in V_\tau$, we define $p_{\sigma,\lambda}^{g,v}\in L^2(K,\sigma)$  by 
  \begin{equation}\label{les-p}
   p_{\sigma,\lambda}^{g,v}(k)  =\sqrt{d_{\tau,\sigma}} e^{(i\lambda-\rho)H(g^{-1}k)}P_\sigma \tau(\kappa(g^{-1}k))^{-1}v.
  \end{equation}
  By the symmetric formula (see e.g. \cite{Campo}, Proposition 3.3]), for any $x\in G$  we have 
 \begin{align}\label{symformula}
 \mathcal{P}_{\sigma,\lambda}^\tau\,(  p_{\sigma,\lambda}^{g,v})(x)=\Phi_{\sigma,\lambda}^\tau(g^{-1}x)v.
\end{align}  
 % \begin{equation}\label{les-p}
  % p_{\sigma,\lambda}^{g,v}(k)  =e^{(i\lambda-\rho)H(g^{-1}k)}P_\sigma \tau(\kappa(g^{-1}k))^{-1}v.
  %\end{equation}

 % Let $w$ a representative in $M'/M$ of the nontrivial element of the Weyl group $W$, where $M'* is the normalizer of $A$ in $K$.

 \begin{lemma}\label{UU} $(1)$ For $\lambda\in\mathbb R\setminus\{0\}$, the set of finite combinations of  $p_{\sigma,\lambda}^{g,v}$\,  $(v\in V_\tau$ and $g\in G)$ is a dense subspace of $L^2(K,\sigma)$.% (see eg. \cite{H2}).

$(2)$ For any $s\in W=\{Id,w\}\simeq\{1,-1\}$ and $\lambda\in\mathbb{R}\setminus\{0\}$, there exists a  unique unitary isomorphism  $U_{s,\lambda}: L^2(K,\sigma)\to L^2(K,s\sigma)$ such that   
  \begin{equation}\label{U-p}
  	  U_{s,\lambda} p_{\sigma,\lambda}^{g,v}=p_{s\sigma,s\lambda}^{g,v}.
  \end{equation}
  Moreover, for $F_1 \in L^2(K,\sigma)$, $F_2 \in L^2(K,s\sigma)$ we have $\mathcal{P}_{\sigma,\lambda}^\tau F_1=\mathcal{P}_{s\sigma,s\lambda}^\tau F_2$ if and only if $U_{s,\lambda}F_1=F_2$ i.e. $U_{s,\lambda}=\left(\mathcal{P}_{s\sigma,s\lambda}^\tau\right)^{-1}\circ \mathcal{P}_{\sigma,\lambda}^\tau$.
  \end{lemma}
\begin{proof}
The assertion $(1)$ follows from the injectiveness of the Poisson transform $\mathcal{P}_{\sigma,\lambda}^\tau$ for $\lambda\in\mathbb{R}\setminus\{0\}$, see Theorem \ref{gaillard}, Corollary \ref{bijective1} and Corollary \ref{cor-olbrich}.

$(2)$ $U_{s,\lambda}$ is well defined in $\operatorname{span}( p_{\sigma,\lambda}^{g,v})$ since $\mathcal P_{s\sigma,s\lambda}^\tau$ remains injective for $\lambda\in\mathbb R\setminus\{0\}$. The identity \eqref{U-p} follows from \eqref{symformula} and the identity 
\begin{align}\label{identityf}
\Phi_{\sigma,\lambda}^\tau=\Phi_{s\sigma,s\lambda}^\tau.
\end{align}
To show that $U_{s,\lambda}$ is unitary, let $g_1,g_2\in G$ and $v_1,v_2\in V_\tau$. 
Then by using  \eqref{symformula} \eqref{identityf} and  noting that   $P_{s\sigma}$ is self-adjoint, we get
 $$\begin{aligned}
&\langle U_{s,\lambda}p^{g_1,v_1}_{\sigma,\lambda}, U_{s,\lambda}p^{g_2,v_2}_{\sigma,\lambda}\rangle_{L^2(K,s\sigma)}\\
&\hspace{2cm}=d_{\tau,\sigma}
\langle\int_K {\rm d}k\, {\rm e}^{(i\lambda-\rho)H(g_1^{-1}k)}{\rm e}^{(-i\lambda-\rho)H(g_2^{-1}k)}\tau(\kappa(g_2^{-1}k)P_{s\sigma}\tau^{-1}(\kappa(g_1^{-1}k)) v_1 ,v_2\rangle\\
&\hspace{2cm}=  \langle\Phi^\tau_{s\sigma,s\lambda}(g_1^{-1}g_2)v_1,v_2 \rangle\\
&\hspace{2cm}=  \langle\Phi^\tau_{\sigma,\lambda}(g_1^{-1}g_2)v_1,v_2\rangle\\
&\hspace{2cm}=\langle p^{g_1,v_1}_{\sigma,\lambda},p^{g_2,v_2}_{\sigma,\lambda}\rangle_{L^2(K,\sigma)},	
\end{aligned}
$$
as to be shown.

\end{proof}

 \begin{theorem}\label{asymp-Poisson}
   Let $\tau=\tau_1,\cdots,\tau_{\frac{n-1}{2}}, \tau^\pm_{\frac{n}{2}}$ and let $\sigma\in\widehat M(\tau)$. For $\lambda\in\mathbb R\setminus\{0\}$ and  $F\in L^2(K,\sigma)$,   we have the following asymptotic expansion for the Poisson transform in $B^*(G,\tau)$,
\begin{equation}\label{asympt30}
 \mathcal P_{\sigma,\lambda}^\tau F(x)
 \simeq \tau(k_2(x))^{-1} \sum_{s\in W} e^{(is\lambda-\rho)A^+(x)} \sqrt{d_{\tau,s\sigma}}\mathbf c(s\lambda,\tau) U_{s,\lambda} F(k_1(x)), 
  \end{equation}
 for any $x=k_1(x)e^{A^+(x)}k_2(x)\in G$.  
  \end{theorem}
 \begin{proof}
   Notice that both sides of \eqref{asympt30} depend continuously on $F\in L^2(K,\sigma)$. Therefore 
in view  of the item (1)   in Lemma \ref{UU}, to  prove    the asymptotic expansion holds in   $L^2(K,\sigma)$ it is sufficient to prove it for the functions    $F=p^{g,v}_{\sigma,\lambda}$ (see \eqref{les-p}).
Since
$$\mathcal P_{\sigma,\lambda}^\tau  p_{\sigma,\lambda}^{g,v}(x) =\Phi_{\sigma,\lambda}^\tau(g^{-1}x) v,$$
we get by Proposition \ref{key-formula}  
$$\mathcal P_{\sigma,\lambda}^\tau  p_{\sigma,\lambda}^{g,v}(x) 
\simeq \sqrt{d_{\tau,\sigma}}  \tau(k_2(x))^{-1}\sum_{s\in W}   e^{(is\lambda-\rho)A^+(x)}  \mathbf{c}(s\lambda,\tau) p_{s\sigma,s\lambda}^{g,v}(k_1(x)),$$
then we use  the identity \eqref{U-p}  to conclude.
\end{proof}

 \section{Proof  of  Theorem \ref{main-th-Poisson} -- Strichartz's conjecture for Poisson transforms }
\noindent {\bf Proof  item (1) of Theorem \ref{main-th-Poisson}}.
 The right hand-side of the estimate \eqref{esti-Poisson16}    follows from Proposition \ref{pro-est-Poisson} and the left hand side form the equality \autoref{asym-Poisson16}.
Now, let us prove \autoref{asym-Poisson16}.  
Let $F\in L^2(K,\sigma)$  and let $\varphi_F$ be the $V_{\tau}$-valued function defined by  %for $g=k_1g) e^{A^+(g)}k_2(g)$,
$$\varphi_F(g)= \sqrt{d_{\tau,\sigma}}\tau(k_2(g))^{-1} 
\left( 
\mathbf{c}(\lambda,\tau) e^{(i\lambda-\rho)A^+(g)} F(k_1(g)) +  
 \mathbf{c}(-\lambda,\tau) e^{(-i\lambda-\rho)A^+(g)} U_{\omega,\lambda} F(k_1(g)) 
 \right).$$
Here $\omega$ is a representative of the non-trivial element of the Weyl group. Using the  change of variable $g=k_1e^{tH_0}k_2$ (see \autoref{chg-cartan}),    the fact that $\tau$ and $U_{\omega,\lambda}$ are unitary,  and $d_{\tau,\sigma}\textbf{c}(\lambda,\tau)P_{\sigma}=c_\sigma(\lambda,\tau)P_{\sigma}$, we have 

 \begin{equation}\label{formule-17-4}
 	\begin{aligned}
\frac{1}{R} &\int_{B(R)} \diff(gK) \| \varphi_F(g)\|^2_\tau  
=  \frac{2}{d_{\tau,\sigma}} |c_\sigma(\lambda,\tau)|^2 \| F\|^2_{L^2(K,\sigma)} \left(\frac{1}{R} \int_0^R  \diff t\,(2e^{-t}\sinh t)^{n-1} \right)+\\
&+ 2  \Re\, \left( c_\sigma(\lambda,\tau)\overline{c_{\omega\sigma}(\omega\lambda,\tau)} \int_K \diff k\,\langle F(k),U_{\omega,\lambda}F(k)\rangle \frac{1}{R}\int_0^R \diff t\,e^{2(i\lambda-\rho)t} (2\sinh t)^{n-1}\right).
\end{aligned}
 \end{equation}
The Riemann-Lebesgue theorem  implies
$$
\lim_{R\to \infty}\frac{1}{R}\int_0^R \diff t\,e^{2(i\lambda-\rho)t} (2\sinh t)^{n-1}=0.
$$
 Taking the limit in \eqref{formule-17-4} and noting that $\frac{2}{d_{\tau,\sigma}}|c_\sigma(\lambda,\tau)|^2=\frac{1}{\pi}\nu_\sigma(\lambda)^{-1}$, we
 get 
\begin{equation}\label{lim}
  \lim_{R\to\infty} \frac{1}{R}  \int_{B(R)} \diff(gK) \| \varphi_F(g)\|^2_{\tau} 
=  \frac{1}{\pi} \nu_\sigma(\lambda,\tau)^{-1}  \|F\|^2_{L^2(K,\sigma)}, 
\end{equation}
%where $\gamma_0$ is a constant depending only on $n$ and $\tau$. 
To conclude, we write 
 $$\begin{aligned}
 \frac{1}{R}\int_{B(R)} \diff(gK) \|\mathcal P_{\sigma,\lambda}^{\tau} F(g) \|^2_{\tau}  &=
 \frac{1}{R}\int_{B(R)}\diff(gK)  \Bigl( \| \varphi_F(g)\|^2_{\tau} +\| \mathcal P_{\sigma,\lambda}^{\tau} F(g)  -\varphi_F(g)\|^2_{\tau} \\
& +  2\Re \langle \mathcal P_{\sigma,\lambda}^{\tau} F(g)  -\varphi_f(g), \varphi_F(g)\rangle_{\tau} \Bigr),
 \end{aligned}
 $$
hence,   \autoref{asym-Poisson16} follows from \autoref{lim}, Theorem \ref{asymp-Poisson} and the Cauchy-Schwarz inequality.\\

It remains to show that $\mathcal P_{\sigma,\lambda}^\tau$ is an isomorphism. Before, we need to establish some intermediate results.
Let $\tau=\tau_1,\cdots,\tau_{\frac{n-1}{2}}, \tau^\pm_{\frac{n}{2}}$ and $\sigma\in\widehat{M}(\tau)$ 
   of dimension $d_{\sigma}$.  Let $\widehat K(\sigma)\subset \widehat K$ be the subset of unitary equivalence classes of irreducible representations containing    $\sigma$ upon restriction to $K$. Consider an element $(\delta,V_\delta)$   in $\widehat{K}(\sigma)$, with $d_\delta=\dim V_{\delta}$.
   From   \cite{BS}  or \cite{IT}  it   follows     that $\sigma$ occurs in $\delta_{|M}$ with multiplicity one, and therefore $\dim \mathrm{Hom}_M(V_\delta, V_\sigma)=1$.  Choose  the orthogonal projection $P_\delta : V_\delta\to V_\sigma$ to be   a generator of $\mathrm{Hom}_M(V_\delta, V_\sigma)$.
 Fix an orthonormal basis    $\{v_j: 1\leqslant j\leqslant d_\delta   \}$ of $V_\delta$. 
 Then  the family 
$\{  \phi^\delta_j: 1\leqslant j\leqslant d_\delta, \; \delta\in \widehat{K}(\sigma)\}$ defined by 
$$k\mapsto \phi^\delta_j(k)=P_\delta(\delta(k^{-1})v_j) $$ 
is  an orthogonal basis of the space $L^2( K, \sigma )$, see, e.g., \cite{wallach}. 
 Hence, the  Fourier series expansion of each   $ F$ in $L^2( K, \sigma)$ is given by   
$$F(k)=\sum_{\delta\in\widehat{K}(\sigma)}\sum_{j=1}^{d_\delta} a^\delta_{j} \phi^{\delta}_j(k),$$
with 
\begin{equation*}%\label{L2-norm}
\displaystyle \Vert F\Vert^2_{L^2(K,\, \sigma)}=\sum_{\delta\in\widehat{K}(\sigma)} \frac{d_\sigma}{d_\delta} \sum_{j=1}^{d_\delta}\mid a^\delta_{j}\mid^2.   
\end{equation*}

For  $\lambda\in\mathbb{C}$ and   $(\delta,V_\delta)\in\widehat K(\sigma)$, define the following  Eisenstein integral 
$$\Phi_{\lambda,\delta}:=\Phi_{\sigma,\lambda}^{\tau,\delta} \colon G\to \operatorname{End}(V_\delta,V_\tau)$$ by
\begin{equation}\label{Eisen-9}
  \Phi_{\lambda,\delta}(g)(v)=  \sqrt{d_{\tau,\sigma}}  \int_K \diff k\,{\rm e}^{-(i\lambda+\rho)H(g^{-1}k)}\tau(\kappa(g^{-1}k))   P_\delta(\delta(k^{-1}) v) ,  
\end{equation}
where $g\in G$ and $v\in V_\delta$.

\begin{proposition}\label{prop-10-aout} We have
 \begin{equation}\label{form-20-aout}
\lim_{R\to\infty} \frac{1}{R}\int_{B(R)} \diff(gK) \left\|\Phi_{\lambda,\delta}(g) \right\|^2_{HS} =\frac{d_\sigma}{\pi}\nu_\sigma(\lambda)^{-1}
\end{equation}
Here $\|\cdot \|_{HS}$ stands for the Hilbert-Schmidt norm.
\end{proposition}
 \begin{proof}
 Since $\Phi_{\lambda,\delta}(g)v_j = \mathcal P_{\sigma, \lambda}^\tau ( P_\delta \delta(\cdot)^{-1}v_j)(g)$ for any $j=1,\cdots, d_\delta$, then by  \autoref{asym-Poisson16}   we have
$$\begin{aligned}
 %\lim_{R\to\infty} \frac{1}{R} 
%&\int_{B(R)} \diff(gK) \left\langle \Phi_{\lambda,\delta}(g) v_j, 
%\Phi_{\lambda,\delta}(g) v_j\right\rangle_{\tau}  \\
%&=
 %\lim_{R\to\infty} \frac{1}{R} 
%\int_{B(R)}  \diff(gK)
%\left\langle 
%\mathcal P_{\sigma,\lambda}^\tau (P_\delta \delta(\cdot)^{-1}v_j)(g), 
%\mathcal P_{ \sigma,\lambda}^\tau (P_\delta \delta(\cdot)^{-1}v_j)(g)
%\right\rangle_{\tau}  \\
%&=2|c_\sigma(\lambda,\tau)|^2  \langle P_\delta %%\delta(\cdot)^{-1}v_j,P_\delta \delta(\cdot)^{-1}%v_j\rangle_\tau.
\lim_{R\to\infty} \frac{1}{R}
\int_{B(R)} \diff(gK) \left\|\Phi_{\lambda,\delta}(g) \right\|^2_{HS} &=\sum_{j=1}^{d_\delta}\lim_{R\to\infty}\frac{1}{R}\int_{B(R)}\diff(gK) \parallel \mathcal P_{\sigma, \lambda}^\tau ( P_\delta \delta(\cdot)^{-1}v_j)(g)\parallel^2 \\
&=\frac{\nu_\sigma^{-1}(\lambda)}{\pi}\sum_{j=1}^{d_\delta}\parallel  P_\delta \delta(\cdot)^{-1}v_j\parallel^2_{L^2(K,\sigma)}\\
&= \frac{\nu_\sigma^{-1}(\lambda)}{\pi}d_\sigma, 
\end{aligned}
$$
as to be shown.
%Therefore
%$$\begin{aligned}
%\lim_{R\to\infty} \frac{1}{R}\int_{B(R)} \diff(gK)\tr\left(\Phi_{\lambda,\delta}(g)^*\Phi_{\lambda,\delta}(g)\right)  
%&= \lim_{R\to\infty} \frac{1}{R}\sum_{j=1}^{d_\delta} \left\langle \Phi_{\lambda,\delta}(g) v_j, 
%\Phi_{\lambda,\delta}(g) v_j\right\rangle_{\tau} \\
%&=2|c_\sigma(\lambda,\tau)|^2 \tr(P_\delta P_\delta^*) \\
%&=2|c_\sigma(\lambda,\tau)|^2 \dim\sigma\\
%&=\frac{d_\tau}{\pi}\nu_\sigma(\lambda)^{-1},
%\end{aligned}
%$$
%by \eqref{density}.
\end{proof}

\noindent {\bf Proof of   item (2) Theorem \ref{main-th-Poisson}.}    We will now prove    that the Poisson transform is a surjective map from $L^2(K,\sigma)$ onto $\mathcal E^2_{\sigma,\lambda}(G,\tau)$. 
 Let $f\in \mathcal E_{\sigma,\lambda}^2(G,\tau)$.
Since $\lambda\in\mathbb R\setminus\{0\}$, then by  Theorem \ref{gaillard}, Corollary \ref{bijective1} and Corollary \ref{cor-olbrich},  there exists $F\in C^{-\omega}(K,\sigma)$ such that $\mathcal P_{\sigma,\lambda}^\tau F=f$. 
From the Fourier expansion  $F(k)=\sum_{\delta\in\widehat K(\sigma )}  \sum_{j=1}^{d_\delta} a^\delta_{j}P_\delta (\delta(k^{-1}))v_j$    we get 
$$
f(g) 
= \sum_{\delta\in\widehat K(\sigma )}  \sum_{j=1}^{d_\delta} a^\delta_{j} \Phi_{\lambda,\delta} (g)  v_j,
$$
then  Schur orthogonality relations show that
$$\begin{aligned}
\int_{B(R)} \diff (gK)\|f(g)\|_\tau^2 =\sum_{\delta\in\widehat K(\sigma )}  \sum_{j=1}^{d_\delta}  \frac{|a_j^\delta|^2}{d_\delta} \int_{B(R)} \diff (gK) \parallel\Phi_{\lambda,\delta}(g)\parallel_{HS}^2 .
\end{aligned}
$$
Therefore,  for any $R>0$, 
$$
\|f\|_*^2\geq \sum_{\delta\in\widehat K(\sigma )}  \sum_{j=1}^{d_\delta}  \frac{|a_j^\delta|^2}{d_\delta} \frac{1}{R}\int_{B(R)} \diff (gK) \parallel\Phi_{\lambda,\tau}(g)\parallel_{HS}^2 
$$
%
%\frac{1}{R}  \int_{B(R)}  \diff(gK)\Vert f(g\Vert_{\tau }^2\, {\rm d} \\
%& \geqslant \frac{1}{R}  \int_0^R \int_K (2\sinh t)^{n-1} \diff t  \diff k\, \Vert f(ka_t)\Vert_{\tau }^2\,  \\
%&\geqslant   \sum_{\delta\in\widehat K(\sigma )}  \sum_{j=1}^{d_\delta}   \frac{|a_j^\delta|^2}{d_\delta} 
%\frac{1}{R}\int_0^R \tr \left(\Phi_{\lambda,\tau}^*(a_t) \Phi_{\lambda,\tau}(a_t) \right) (2\sinh t)^{n-1} \diff t.
%\end{aligned}
%$$
In the above inequality  the summation over $\delta\in \widehat K(\sigma)$   remains true in particular over $\delta\in \Pi$ for any finite subset $\Pi \subset \widehat K(\sigma)$. 
Thus by \autoref{form-20-aout}  we get,  
$$
\|f\|_*^2\geqslant  = \frac{\nu_\sigma(\lambda)^{-1}}{\pi} \sum_{\delta\in \Pi }\sum_{j=0}^{d_\delta} \frac{d_\sigma}{d_\delta} |a_j^\delta|^2.
$$
Since $\Pi $ is arbitrary, we obtain
$$
 \frac{\nu_\sigma(\lambda)^{-1}}{\pi}
\sum_{\delta\in\widehat K(\sigma)}  \sum_{j=1}^{d_\delta}  \frac{d_{\sigma} }{d_\delta}|a_j^\delta|^2\leqslant \|f\|_*^2<\infty,$$
which proves that $F\in L^2(K,\sigma)$ and 
$$
\frac{\nu_\sigma(\lambda)^{-1}}{\pi}\parallel F\parallel^2_{L^2(K,\sigma)}\leq \parallel f\parallel_\ast^2.
$$
This completes the proof of   Theorem \ref{main-th-Poisson}.

\section{Proof of Theorem \ref{inversion} -- Inversion formula for the Poisson transform}
 Let $f\in \mathcal{E}^{2}_{\sigma,\lambda}(G, \tau)$, then by Theorem \ref{main-th-Poisson}, there exists a unique $F \in L^2(K,\sigma)$ such that $f=\mathcal P_{\sigma,\lambda}^\tau F$. Now define the function $F_R$   by
 $$F_R(k)= \pi\nu_\sigma(\lambda) \frac{1}{R}\int_{B(R)} e^{(i\lambda-\rho)H(g^{-1}k)} \tau^{-1}(\kappa(g^{-1}k)f(g)\diff (gK),\;\; \text{for any $R\geqslant0$},$$
 and consider its Fourier expansion 
 $$F_R(k)=\sum_{\delta\in\widehat{K}(\sigma)}\sum_{j=1}^{d_\delta} c_j^\delta(R)\sqrt{\frac{d_\delta}{d_\sigma}} P_\delta (\delta(k^{-1})v_j)$$
where the coefficients are given by
$$c_j^\delta(R)=\sqrt{\frac{d_\delta}{d_\sigma}}\langle F_R,\overline{P_\delta \delta^{-1}(\cdot)v_j}\rangle$$
Above we have used the orthonormal basis $k\to \sqrt{\frac{d_\delta}{d_\sigma}}P_\delta(\delta(k^{-1})v_j)$ of $L^2(K,\sigma)$ instead of the orthogonal basis from  Section 7. 

Let us show that $F_r\to F$ as $R\to \infty$ in $L^2(K,\sigma)$. To do so, it is enough to prove that $\lim_{R\to\infty} \|F_R\|_{L^2(K,\sigma)}=\|F\|_{L^2(K,\sigma)}$.

From the expansion of $F$, $F(k)=\sum_{\delta'\in\widehat{K}(\sigma)}\sum_{\ell=1}^{d_{\delta'}} a_\ell^{\delta'} \sqrt{\frac{d_{\delta'}}{d_\sigma}} 
P_{\delta'}(\delta'(k^{-1})v_j)$
and since $f=\mathcal P_{\sigma,\lambda}^\tau F$, we get
$$f(g)=\sum_{\delta'\in\widehat{K}(\sigma)}\sum_{\ell=1}^{d_{\delta'}} a_\ell^{\delta'}\sqrt{\frac{d_{\delta'}}{d_\sigma}}  \Phi_{\lambda,\delta'}(g)v_\ell
$$
where $\Phi_{\lambda,\delta'}$ is given by \eqref{Eisen-9}.
Thus,
$$F_R(k)=\pi\nu_\sigma(\lambda)\sum_{\delta',\ell}
\sqrt{\frac{d_{\delta'}}{d_\sigma}}
a_\ell^{\delta'}\frac{1}{R}\int_{B(R)} \diff (gK) \,e^\tau_{\sigma,\lambda}(k^{-1}g)
\Phi_{\lambda,\delta'}(g)v_\ell.$$
Therefore
$$
c_j^\delta(R)=\pi\nu_\sigma(\lambda)\sum_{\delta',\ell} a_{\ell}^{\delta'} \frac{\sqrt{d_{\delta'} d_{\delta}}}{d_\sigma}	\int_{B(R)} \diff (gK)\, 
\langle \Phi_{\lambda,\delta'}(g)v_\ell, \Phi_{\lambda,\delta}(g)v_j\rangle  
$$
which, by Schur's orthogonality relation, can be written as
$$c_j^\delta(R)=\pi\nu_\sigma(\lambda)  
\frac{a_j^\delta  }{d_\sigma} \frac{ 1}{R} \int_{B(R)} \diff (gK) \|\Phi_{\lambda,\delta}(g)\|^2_{\operatorname{HS}}.
$$
Next, using   Proposition \ref{prop-10-aout}, we get $\lim_{R\to\infty} c_j^\delta(R)=a_j^\delta$, thus
\begin{equation}\label{10-out}
\lim_{R\to\infty}\|F_R\|_{L^2(K,\sigma)}=\lim_{R\to\infty} \sum_{\delta,j} |c_j^\delta(R)|^2=\sum_{\delta,j}|a_j^\delta|^2=\| F\|^2_{L^2(K,\sigma)}.	
\end{equation}
To justify interchanging the sum and limit signs in \eqref{10-out}, we would note that
 \begin{eqnarray*}
	\frac{1}{R}\int_{B(R)} \diff (gK) \|\Phi_{\lambda,\delta}(g)\|_{HS}^2 &=& \sum_j
\frac{1}{R} \int_{B(R)}\diff (gK)\, \|\mathcal P_{\sigma,\lambda}^\tau \left(P_\delta \delta^{-1}(\cdot) v_j\right)(g) \|^2\\
&\leqslant& \sum_j \| \mathcal P_{\sigma,\lambda}^\tau \left(P_\delta \delta^{-1}(\cdot) v_j\right)\|_*^2\\
&\leqslant& C \nu_\sigma(\lambda)^{-1}\sum_j \| P_\delta \delta^{-1}(\cdot) v_j \|^2_{L^2(K,\sigma)}\quad\quad   \text{by \eqref{esti-Poisson16} }\\
&\leqslant& C \nu_\sigma(\lambda)^{-1}\sum_j \frac{d_\sigma}{d_\delta} = C \nu_\sigma(\lambda)^{-1}d_\sigma.
\end{eqnarray*}
Hence
$$|c_j^\delta(R)|\leqslant \pi C\mid a^\delta_j\mid,$$
 for any $\delta$ and any $j$. This completes the proof of Theorem \ref{inversion}.

    \section{Proof  of  Theorem \ref{main-th-proj} -- Strichartz's conjecture for spectral projections}  
     
  %We will give the proof for $n$ odd. The case when $n$ is even can be proved in the same way.
%We will focus on the proof in the generic case. The proof for the special cases is similar.

\noindent{\bf Proof of   item (1) of Theorem \ref{main-th-proj}.}     
Assume that $p\neq \frac{n}{2}$  and $\tau=\tau_p$.
 Let $f\in L^2(G,\tau)$, and define $\mathcal Q^\tau_\sigma f$   on $(0,+\infty)\times G$ by $\mathcal Q^\tau_\sigma f(\lambda,g)=\mathcal Q^\tau_{\sigma,\lambda} f(g)$.  Then  
 %Since
 %$\mathcal{Q}_{\sigma, \lambda}^\tau f=   {\nu(\lambda)}   \mathcal{P}_{\sigma, \lambda}^\tau\left(\mathcal{F}_{\sigma, \lambda}^\tau f\right)$
 %for a.e.  $\lambda \in \mathbb{R}\setminus\{0\}$, the uniform estimate \autoref{esti-Poisson-odd} of the Poisson transform    implies
 by definition
\begin{align*}
\parallel \mathcal{Q}^\tau f\parallel^2_+&=\sum_{\sigma\in \widehat{M}(\tau)}\parallel \mathcal{Q}_\sigma^\tau f\parallel^2_+.
\end{align*} 
Using the identity  $\mathcal{Q}_{\sigma, \lambda}^\tau f=   {\nu(\lambda)}   \mathcal{P}_{\sigma, \lambda}^\tau\left(\mathcal{F}_{\sigma, \lambda}^\tau f\right)$ and  Theorem \ref{main-th-proj} we get
\begin{align*}%\begin{split}
\parallel \mathcal{Q}_\sigma^\tau f\parallel^2_+&=\sup_{R>0}
\int_0^\infty {\rm d}\lambda \left(\frac{1}{R} \int_{B(R)}\diff(gK)\left\|\mathcal{Q}_{\sigma,\lambda}^\tau f(g)\right\|_\tau^2\right) \\
&\leqslant \int_0^\infty  {\rm d}\lambda \parallel \mathcal{Q}_{\sigma,\lambda}f\parallel^2_\ast \\
&\leqslant C \int_0^\infty \nu_\sigma(\lambda)  \diff \lambda \parallel \mathcal{F}_{\sigma, \lambda}^\tau f\parallel^2_{L^2(K,\sigma)}.
%\end{split}
\end{align*}
Thus 
\begin{align*}
\parallel \mathcal{Q}^\tau f\parallel^2_+
\leqslant C  \sum_{\sigma\in \widehat{M}(\tau)}\int_0^\infty \nu_\sigma(\lambda)\diff \lambda\parallel \mathcal{F}_{\sigma, \lambda}^\tau f\parallel^2_{L^2(K,\sigma)}= C \parallel f\parallel^2_{L^2(G,\tau)},
\end{align*}
by the Plancherel formula. This proves the right hand side of \autoref{esti-main1}.

Furthermore, by \autoref{asym-Poisson16} and  \autoref{esti-Poisson16}  we have  
$$
 \lim_{R\to\infty}\frac{1}{R} \int_{B(R)} \diff(gK)\left\|\mathcal{Q}_{\lambda}^{\tau} f(g)\right\|_{\tau}^2  
=  \frac{1}{\pi} \sum_{\sigma\in \widehat{M}(\tau)}
  \nu_\sigma(\lambda) \diff \lambda   \| \mathcal {F}_{\sigma,\lambda}^{\tau}f   \|^2_{L^2(K,\sigma)},
$$
and  
$$ 
\begin{aligned}
\frac{1}{R} \int_{B(R)}\diff(gK)\left\|\mathcal{Q}_{\lambda}^{\tau} f(g)\right\|_{\tau}^2  &\leqslant 
\sum_{\sigma\in \widehat{M}(\tau)}
  \nu_\sigma(\lambda) \diff \lambda  \| \mathcal {F}_{\sigma,\lambda}^{\tau}f   \|^2_{L^2(K,\sigma)}.	
\end{aligned}
$$
Then by the Lebesgue dominated convergence and the Placherel formula \autoref{Planch-generic} we get
$$ \begin{aligned}
\lim_{R\to\infty}\frac{1}{R}  \int_0^\infty \diff\lambda\int_{B(R)}\diff(gK)\left\|\mathcal{Q}_{\lambda}^{\tau} f(g)\right\|_{\tau}^2 
&= \frac{1}{\pi} \sum_{\sigma\in \widehat{M}(\tau)}\int_0^\infty \nu_\sigma(\lambda) \diff \lambda      \| \mathcal F_{\sigma,\lambda}^{\tau} f   \|^2_{L^2(K,\sigma)}  \\
&=\frac{1}{\pi}\|f\|^2_{L^2(G,\tau)},
\end{aligned}$$
which prove \autoref{asym-Q} and  indeed the left hand side of \autoref{esti-main1}. 

 The proof for $p=\frac{n}{2}$ is similar to the above taking in account $L^2(G,\tau^\pm_{\frac{n}{2}})_{\mathrm{cont}}$.

\noindent{\bf Proof of item (2) of Theorem \ref{main-th-proj}.} 
 Assume that $p\neq \frac{n}{2}$ and $\tau=\tau_p$. Linearity together with  \autoref{esti-main1} imply the injectivity of  $\mathcal Q^{\tau}$. It remains then to prove that $ \mathcal Q^{\tau}$ is onto. 
 Let $\psi \in\mathcal E^2_{\mathbb R_>}(G,\tau)$. Then for any $\lambda$, $\psi (\lambda,\cdot) =\sum_{\sigma\in \widehat{M}(\tau)}\psi_{\sigma,\lambda}$ with $\psi_{\sigma,\lambda}\in 
 \mathcal E_{\sigma,\mathbb R_>}^2(G,\tau)$.
Since
$$\sup_{R>0} \frac{1}{R}\int_{B(R)} \diff(gK) \|\psi_{\sigma,\lambda} (g)\|_{\tau}^2  <\infty \quad \text{a.e.}\; \lambda\in (0,\infty),$$
then applying    Theorem \ref{main-th-Poisson}  we can assert the existence of $F_{\sigma,\lambda}\in L^2(K,\sigma)$  such that 
for a.e. $\lambda$ we have $\psi_{\sigma,\lambda}=\nu_\sigma(\lambda)\mathcal P^{\tau}_{\sigma,\lambda} F_{\sigma,\lambda}$.
Furthermore, it follows from the estimate \eqref{esti-Poisson16} that
$$\begin{aligned}
\sup_{R>0}\frac{1}{R}  \int_{B(R)} \diff(gK)\|\psi_{\sigma,\lambda} (g)\|_{\tau}^2   
&\geqslant
C^{-2} \nu_\sigma(\lambda)\int_K \diff k\,\|F_{\sigma,\lambda}(k)\|^2_{\sigma},
\end{aligned}
$$
hence
$$ 
C^{-2}\int_0^\infty   \int_K \diff k\|F_{\sigma,\lambda}(k)\|^2_{\sigma} \nu_\sigma(\lambda)\diff\lambda \leqslant \|\psi_{\sigma,\lambda} \|_+^2 <\infty,$$
 which proves that 
 $$F_\lambda=\sum_{\sigma\in\widehat{M}(\sigma)} F_{\sigma,\lambda} \in \oplus_{\sigma\in\widehat{M}(\sigma)} L^2(\R_+ ;L^2(K,\sigma), \nu_\sigma(\lambda)\diff\lambda).$$ 
  Now by the Plancherel theorem (Theorem \ref{pedon}), there exists $f\in L^2(G,\tau)$ such that 
  $F_\lambda=\mathcal F^{\tau} f(\lambda)=(\mathcal F_{\sigma,\lambda}^{\tau} f)_\sigma$, 
  thus, for any $g\in G$,
  $$\begin{aligned}
\psi(\lambda,g) &= \sum_{\sigma\in\widehat{M}(\sigma)}\nu_\sigma(\lambda) \mathcal P_{\sigma,\lambda}^{\tau} \left( \mathcal F_{\sigma,\lambda}^{\tau} f \right)(g) 
 \\
 & =\mathcal Q_\lambda^{\tau}( f)(g)=\mathcal Q^{\tau} f(\lambda,g),
 \end{aligned}$$
 as required to be shown.  \\
The case $p=\frac{n}{2}$ can be treated similarly. %can be handled in the same way. 
%Again, if $p=\frac{n}{2}$, the proof is similar to the above.
  This completes the proof of Theorem \ref{main-th-proj}.

  \section{Appendix}
   In this appendix we will provide the proof   lemma \ref{key-fo11} and Lemma \ref{cliff}, we have used to show the key formula in Proposition \ref{key-formula}. For the convenience of the reader, we will recall their statement.

%Now we shall prove the two lemmas  using the above realization. For the convenience of the reader, we will recall their statement.

Elements of $G=\SO_0(n,1)$ are of the form 
$g=\begin{pmatrix}
	A&b   \\
	c^\top & d
\end{pmatrix}$ 
 where $A\in \mathcal M(n,\mathbb R)$,   $b, c \in \mathbb R^n$  (column vectors)   and $d>0$ a real number, with the relations
 \begin{equation}\label{rel-so(n-1)}
\begin{aligned}
	AA^\top &= I_n+cc^\top ,\\
	A^\top b &=dc,\\
	b^\top b  &= d^2-1.
\end{aligned}	
\end{equation}

The Cartan decomposition of the Lie algebra $\mathfrak{so}(n,1)$ is $\mathfrak{so}(n,1)= \mathfrak{k} \oplus \mathfrak{p}$, where
$$\begin{aligned}
	\mathfrak{k}&=\left\{\begin{pmatrix}
x&0\\
0&0	
\end{pmatrix} \; : x \in\mathcal M_n(\mathbb R),\; x^\top=-x
	\right\},\\
	\mathfrak{p}&=\left\{\begin{pmatrix}
0_n&y\\
y^\top&0	
\end{pmatrix} \; : y\in\mathbb R^n \, (\text{column})
	\right\}.
\end{aligned}$$
The tangent space of $G/K$ at $eK$ is identified to $\mathfrak p$. %The analytic subgroup $K$ of $G$ correponding to $\mathfrak{k}$ is

Fix    the Cartan subspace $\mathfrak{a}=\mathbb R H_0$ of  $\mathfrak{p}$, where
$$H_0=
\begin{pmatrix}
 0_n & e_1\\ 
e_1^\top & 0  	
\end{pmatrix}\in \mathfrak{p},$$
and $e_1=(1,0,\ldots,0)^\top\in \mathbb R^n$.  The 
  subgroups $K$, $A$ and $N$ are given by
$$
\begin{aligned}
K&=\left\{k=\begin{pmatrix}
u&0\\
0&1	
\end{pmatrix} \; : u\in\SO(n)  
	\right\},
\\
A&=
\left\{a_t=e^{tH_0}=\begin{pmatrix}
\cosh(t) &0& \sinh(t)\\
0&I_{n-1}&0\\
\sinh(t)&0&\cosh(t)
\end{pmatrix} \; : t\in\mathbb R  
	\right\},
\\
N&=
\left\{n_y=\begin{pmatrix}
1-\frac{1}{2}\|y\|^2 &y^\top& \frac{1}{2}\|y\|^2\\
-y&I_{n-1}&y\\
-\frac{1}{2}\|y\|^2&y^\top&1+\frac{1}{2}\|y\|^2
\end{pmatrix} \; : y\in\mathbb R^{n-1} \, (\text{column})  
	\right\}.
	\end{aligned}
$$
 Recall that the action of $K$ on the tangent space corresponds to the adjoint action of $K$ on $\mathfrak p$ : 
if $k=\begin{pmatrix}
	u&0\\ 
	0&1
\end{pmatrix}\in K$ 
and 
$Y=\begin{pmatrix}
	0 & y\\
	y^\top & 0
\end{pmatrix}\in \mathfrak p$, then $Ad(k) Y=\begin{pmatrix}
	0 & uy\\
	(uy)^\top & 0
\end{pmatrix}$. Thus, the adjoint action of $K$ on $\mathfrak p$ is identified with the linear action of $\SO(n)$ on $\mathbb R^n$.
	
	For any $g\in G$, we have the Cartan and the Iwasawa decompositions,
	$$\begin{aligned}
	g&=k_1(g)e^{A^+(g)}k_2(g)\in K\overline{A^+}K\\
	&=\kappa(g)e^{H(g)}n(g)\in KAN.
	\end{aligned}
	$$
	Then a direct calculation yields\footnote{Notice that from the relation \eqref{rel-so(n-1)}, $d^2-1$ is positive },  
\begin{equation}\label{decompositions-a}
H(g)=\log \mid c^\top e_1+d\mid \;\;\text{and}\;\; A^+(g)=\log\left(d  +\sqrt{d^2-1}\right),
\end{equation}
where $g=\begin{pmatrix}
	A&b   \\
	c^\top & d
\end{pmatrix}$. 

\begin{lemmab}\label{cliffA}   For any $g, x\in G$ we have
$$
A^+(g^{-1}x)=A^+(x)+H(g^{-1}k_1(x))+E(g,x),
$$
where the function $E$ satisfies the estimate
$$
0<E(g,x)\leqslant {\rm e}^{2(A^+(g)-A^+(x))}.
$$
\end{lemmab}

\begin{proof}

%Let $g=\begin{pmatrix}
%a&b\\c&d
%\end{pmatrix}\in G$, with $a\in\mathcal{M}_{n,n}(\mathbb R)$, $b\in\mathcal{M}_{1,n}(\mathbb R)$,  $c\in\mathcal{M}_{n,1}(\mathbb R)$  and $d\in\mathbb R$. Then a direct calculation yields,  
%\begin{equation}\label{decomposition}
%H(g)=\log \mid ce_1+d\mid \quad A^+(g)=\log\left(\mid d\mid +\sqrt{\mid d\mid^2-1}\right),
%\end{equation}
%where $e_1=(1,0,\cdots,0)^\top$. 
Put $E(g,x)=A^+(g^{-1}x)-A^+(x)-H(g^{-1}k_1(x))$ and write $x=k{\rm e}^{tH_0}h$ with respect to the Cartan decomposition. Then we have 
$$
E(g, k{\rm e}^{tH_0}h)=A^+(g^{-1}k{\rm e}^{tH_0})-H(g^{-1}k)-t.
$$
Next write $g^{-1}=\begin{pmatrix}
A&b\\c^\top&d
\end{pmatrix}$ and  $k=\begin{pmatrix}
u&0\\0&1
\end{pmatrix}$. Then according to \eqref{decompositions-a},  
$$
\begin{aligned}
&E(g, k{\rm e}^{tH_0}h)\\
&=\log\left(  | \sinh(t) c^\top   u\,e_1+d\cosh(t)| +\sqrt{\mid \sinh(t) c^\top u\, e_1+d\cosh(t)\mid^2-1}\right)
-\log | c^\top u\, e_1+d |-t.	
\end{aligned}
 $$
Thus
\begin{align*}
E(g, k{\rm e}^{tH_0}h)&\leqslant \log \left(1+{\rm e}^{-2t}\frac{\mid c^\top  u\, e_1-d\mid}{\mid c^\top  u\, e_1+d\mid}\right)\\
&\leqslant \log\left(1+{\rm e}^{-2t}{\rm e}^{2A^+(g)}\right)
\end{align*}
and the result follows.

\end{proof}

For $g\in G$, we denote by $\pi_o(g)$ its $K$-component in the  polar decomposition $G=K\exp(\mathfrak p)$.  Then, if $g=\begin{pmatrix}
	A&b\\ c^\top & d
\end{pmatrix}$, one can prove that
\begin{equation}
	\pi_0(g)=
\begin{pmatrix}\label{polar-dec}
	A- \frac{1}{1+d}  b c^\top&0   \\
	0 & 1
\end{pmatrix}.
\end{equation}
 Recall  that, if $g=k_1(g)e^{tH_0}k_2$ is  the Cartan decomposition of $g$, then
 \begin{equation}\label{rela-cart-pol}
 	\pi_0(g)=k_1(g)k_2(g).
 \end{equation}

\begin{lemmab}\label{key-fo11A}
  For any $g\in G$  we have
\begin{equation}\label{lm1-fo11A}
\lim_{R\to\infty}   \tau^{-1}(k_2(g e^{RH_0})P_\sigma\tau^{-1}(k_1(g e^{RH_0})=P_\sigma\tau^{-1}(\kappa(g)).
 \end{equation}
  \end{lemmab}

   \begin{proof} %Recall that 
	%the representation $\tau_p$   on $\bigwedge^p \mathbb C^n$ is given by
%$$
%\tau_p(k)\left(v_1 \wedge \cdots \wedge v_p\right)=k v_1 \wedge \cdots \wedge k v_p (=Ad(k) v_1 \wedge \cdots \wedge Ad(k) v_p
%$$
{\bf Case $\tau=\tau_p$ with   $p$ generic.}
Set $\mathbb C^n=\operatorname{span}(e_1,e_2,\cdots,e_n)$ and $\mathbb C^{n-1}=\operatorname{span}(e_2,\cdots,e_n)$, then  
	${\tau_p}_{|M}=\sigma_{p-1}\oplus \sigma_p$ with de decomposition of the representation space 
	$$\bigwedge\nolimits^p\mathbb C^n=e_1\wedge \left(\bigwedge\nolimits^{p-1}\mathbb C^{n-1}\right) \oplus \bigwedge\nolimits^{p}\mathbb C^{n-1}.$$
	 The projection on $\sigma_p$ is given by $P_{\sigma_p} \xi=e_1\wedge \xi = \varepsilon_{e_1} \xi$ (the left exterior product by $e_1$) and the projection on $\sigma_{p-1}$ is given by $P_{\sigma_{p-1}} \xi=\iota_{e_1} \xi=\langle \xi,e_1\rangle$   (the interior  product by $e_1$).
	  
Since
$$
\tau_p(k) \varepsilon_v \tau_p\left(k^{-1}\right)=\varepsilon_{k v}\;\; \text{and}\;\; \tau(k) \iota_{v} \tau_p\left(k^{-1}\right)=\iota_{kv} ,
$$ 
	  we have
	 $$P_{\sigma_p}\tau_p(k^{-1})\xi=\tau_p(k^{-1})(\varepsilon_{ke_1}\xi)\;\; \text{and}\;\; P_{\sigma_{p-1}}\tau_p(k^{-1})\xi=\tau_p(k^{-1})(\iota_{ke_1}\xi).$$
	 %where $\varepsilon_v w=v\wedge w$ is the exterior  product.
	 Thus, in the 	case $q=p$, 
	$$
	\begin{aligned}
& 
\tau_p^{-1}\left(k_2\left(g e^{R H_0}\right)\right) P_{\sigma_{p}} \tau_p^{-1}\left(k_1\left(g e^{R H_0}\right)\right) \xi
-P_{\sigma_{p}} \tau_p^{-1}(\kappa(g))\xi
  \\
&=
\tau_p^{-1}\left(k_2\left(g e^{R H_0}\right)\right)   \tau_p^{-1}\left(k_1\left(g e^{R H_0}\right)\right) \varepsilon_{k_1\left(g e^{R H_0}\right)e_1}\xi
-  \tau_p^{-1}(\kappa(g))\varepsilon_{\kappa(g)e_1}\xi
 .
\end{aligned}
$$ 
 Using \eqref{rela-cart-pol} we have $ k_2\left(g e^{R H_0}\right)     k_1\left(g e^{R H_0} \right)= \pi_0(ge^{RH_0}) $, and from the Iwasawa decomposition  $g=\kappa(g) a_t n_y$, we get  $\pi_0(ge^{RH_0})=\kappa(g) \pi_0(e^{tH_0} n_y e^{RH_0})$. But
 \begin{equation}
 e^{tH_0} n_y e^{RH_0}=
 \begin{pmatrix}\label{deco-ana}
 	\cosh(t+R) -\frac{e^{t-R}}{2}\|y\|^2 & e^t y^\top &\sinh(t+R) +  \frac{e^{t-R}}{2}\|y\|^2\\
 	-e^{-R}y & I_{n-1} & e^{-R} y\\
 	\sinh(t+R) -\frac{e^{t-R}}{2}\|y\|^2 & e^t y^\top &\cosh(t+R) +  \frac{e^{t-R}}{2}\|y\|^2
 \end{pmatrix}.
 \end{equation}
Hence, if $\pi_0(e^{tH_0} n_y e^{RH_0})=\begin{pmatrix}
 	B&0\\
 	0&1
 \end{pmatrix}$, then  using \eqref{polar-dec} we obtain  
$$\begin{aligned}
B&=\begin{pmatrix}
\cosh(t+R)-\frac{e^{t-R}}{2}\|y\|^2 & e^t y^\top  \\
-e^{-R} y & I_{n-1}
\end{pmatrix} -\frac{1}{1+\cosh(t+R)+\frac{e^{t-R}}{2}\|y\|^2}\times\\
&
\hspace{3cm}\begin{pmatrix}
\sinh^2(t+R)-\frac{e^{2(t-R)}}{4}\|y\|^4 & e^t\left(\sinh(t+R)+\frac{e^{t-R}}{2}\|y\|^2\right)y^\top\\
e^{-R}\left(\sinh(t+R)-\frac{e^{t-R}}{2}\|y\|^2\right)y & e^{t-R} y y^\top
\end{pmatrix}
\end{aligned}
$$
and one can check easily that $B\to I_{n-1}$ when $R\to +\infty$.

On the other hand, $k_1(ge^{RH_0})e_1 =\kappa(g) k_1(e^{t H_0} n_y e^{RH_0})e_1$. Then from \eqref{deco-ana} we get

\begin{eqnarray*}
	k_1(e^{t H_0} n_y e^{RH_0})e_1&=&
\left(\left(\cosh(t+R)+\frac{e^{r-R}}{2} \|y\|^2\right)^1-1\right)^{-\frac{1}{2}} \begin{pmatrix}
	\sinh(t+R) +\frac{e^{t-R}}{2} \|y\|^2\\
	e^{-R} y 
\end{pmatrix} \\
&&\xrightarrow[R \to \infty]{} e_1
\end{eqnarray*}
and therefore, $k_1(ge^{RH_0})e_1 \xrightarrow[R \to \infty]{} \kappa(g)e_1$. This proves \eqref{lm1-fo11A} for $q=p$ in the generic case. The case $q=p-1$ is similar to the above,   and it is left to the reader. 
 
 {\bf Case $\tau_p=\tau_{\frac{n}{2}}^\pm$}. In this case ${\tau_{\frac{n}{2}}^\pm}_{|M}=\sigma_{\frac{n}{2}}$ and the projection $P_\sigma$ on $\sigma$ is the identity. Then \eqref{lm1-fo11A} follows from the proof for $p$  the generic and $q=p$.
 
 {\bf Case $\tau_p=\tau_{\frac{n-1}{2}}$}. In this case 
  ${\tau_{p}}_{|M}=\sigma_{p-1}\oplus\sigma_{p}^+\oplus \sigma_{p}^-$, and the representation space is decomposed as
  $$
\extp^{\frac{n-1}{2}}\mathbb C^n=e_1 \wedge (\extp\nolimits^{\frac{n-1}{2}-1} \mathbb{C}^{n-1}) \oplus \extp\nolimits_{+}^{\frac{n-1}{2}} \mathbb{C}^{n-1} \oplus \extp\nolimits_{-}^{\frac{n-1}{2}} \mathbb{C}^{n-1}.
$$
The projection $P_\sigma$ on $\sigma_{p-1}$
 is $\xi\to \varepsilon_{e_1}\xi=e_1\wedge \xi$ and \eqref{lm1-fo11A} follows from the proof for $p$ generic.\\
 On the other hand, the projection $P_{\sigma_p^\pm}$ on $\sigma_{p}^\pm$ is given by
 $$P_{\sigma_p^\pm}\xi =\frac{1}{2} e_1\wedge \xi \pm \frac{1}{2} i^{p(p+2)}\star (e_1\wedge \xi)=\frac{1}{2} \varepsilon_{e_1} \xi \pm \frac{1}{2} i^{p^2+2}\iota_{e_1} (\star\xi).$$
   Again \eqref{lm1-fo11A} follows from the proof for $p$ generic.
\end{proof}

We hereby announce that the  Strichartz conjecture for the bundle of spinors over $H^n(\mathbb R)$ is discussed in the forthcoming paper \cite{BK-Spin}.
  %____________
 \pdfbookmark[1]{References}{ref}

\end{document}